\documentclass[10pt]{amsart}

\usepackage{amsmath,amsfonts,amssymb,mathrsfs,tensor,bbold}
\usepackage{pinlabel,multirow}

\DeclareMathAlphabet\oldmathcal{OMS}        {cmsy}{b}{n}
\SetMathAlphabet    \oldmathcal{normal}{OMS}{cmsy}{m}{n}
\DeclareMathAlphabet\oldmathbcal{OMS}       {cmsy}{b}{n}

\usepackage{eucal}

\theoremstyle{plain}

\newtheorem{thm}{Theorem}[section]
\newtheorem{lem}[thm]{Lemma}
\newtheorem{prop}[thm]{Proposition}
\newtheorem{cor}[thm]{Corollary}
\newtheorem{defn}[thm]{Definition}

\newenvironment{remark}{\medskip \refstepcounter{thm}
\noindent  {\bf Remark \thethm}.\rm}{\,}

\newtheorem{thmint}{Theorem}

\newtheorem{corint}[thmint]{Corollary}

\newcounter{num}

\newenvironment{thmlist}{\begin{list}{(\roman{num})}{\usecounter{num}\setlength{\leftmargin}{25pt}
\setlength{\itemindent}{0pt}\setlength{\labelwidth}{20pt}\setlength{\labelsep}{5pt}\setlength{\itemsep}{0in}}}{\end{list}}

\makeatletter
\def\amsbb{\use@mathgroup \M@U \symAMSb}
\makeatother

\newcommand{\Z}{\mathbb{Z}}

\newcommand{\R}{\mathbb{R}}
\newcommand{\C}{\mathbb{C}}
\newcommand{\Ha}{\mathbb{H}}
\newcommand{\re}{\operatorname{Re}}

\newcommand{\rps}{\mathbb{RP}}
\newcommand{\cps}{\mathbb{CP}}
\newcommand{\qps}{\mathbb{HP}}

\newcommand{\Aut}{\operatorname{Aut}}
\newcommand{\hol}{\operatorname{\mathfrak{hol}}}

\newcommand{\Hol}{\operatorname{Hol}}

\newcommand{\End}{\operatorname{End}}
\newcommand{\Isom}{\operatorname{Isom}}
\newcommand{\isom}{\operatorname{\mathfrak{isom}}}
\newcommand{\im}{\operatorname{Im}}
\newcommand{\Ric}{\operatorname{Ric}}
\newcommand{\Ricci}{\operatorname{Ricci}}

\newcommand{\tr}{\operatorname{tr}}
\newcommand{\Span}{\operatorname{Span}}
\newcommand{\Sym}{\operatorname{S}}

\newcommand{\ED}{\operatorname{\mathcal{ED}}}
\newcommand{\TED}{\operatorname{\mathcal{TED}}}
\newcommand{\EED}{\operatorname{\mathcal{EED}}}
\newcommand{\PM}{\operatorname{\mathscr{PM}}}
\newcommand{\EDS}{\operatorname{\mathscr{ED}}}

\newcommand{\GL}{\operatorname{GL}}
\newcommand{\U}{\operatorname{U}}
\newcommand{\SU}{\operatorname{SU}}
\newcommand{\Or}{\operatorname{O}}
\newcommand{\SO}{\operatorname{SO}}
\newcommand{\Sp}{\operatorname{Sp}}
\newcommand{\G}{\operatorname{G}}
\newcommand{\Spin}{\operatorname{Spin}}
\newcommand{\Pin}{\operatorname{Pin}}
\newcommand{\Cl}{\operatorname{Cl}}
\newcommand{\CCl}{\operatorname{\mathbb{C}l}}
\newcommand{\SR}{\operatorname{S}^2}

\newcommand{\LSp}{\operatorname{\mathfrak{sp}}}

\newcommand{\ol}[1]{\mkern 1.5mu\overline{\mkern-1.5mu#1\mkern-1.5mu}\mkern 1.5mu}
\newcommand{\contr}{\,\lrcorner\,}

\newcommand{\simarrow}{\stackrel{\sim}{\longrightarrow}}

\title[Deformations of Killing spinors]{Deformations of Killing spinors on Sasakian and 3-Sasakian manifolds}
\author{Craig van Coevering}
\address{Department of Mathematics 15-01, U.S.T.C., Anhui, Hefei 230026, P. R. China}
\email{craigvan@ustc.edu.cn}
\date{March 26, 2015}

\begin{document}

\begin{abstract}
We consider some natural infinitesimal Einstein deformations on Sasakian and 3-Sasakian manifolds.
Some of these are infinitesimal deformations of Killing spinors and further some integrate to actual Killing spinor deformations.  In particular, on 3-Sasakian 7 manifolds these yield infinitesimal Einstein deformations
preserving 2, 1, or none of the 3 independent Killing spinors.  Toric 3-Sasakian manifolds provide non-trivial examples with
integrable deformation preserving precisely 2 Killing spinors.  Thus in contrast to the case of parallel spinors the dimension
of Killing spinors is not preserved under Einstein deformations but is only upper semi-continuous.
\end{abstract}

\maketitle

\section*{Introduction}

Let $M$ be an n-dimensional Riemannian spin manifold with spinor bundle $\Sigma$.  A Killing spinor is a non-trivial
section $\psi\in\Gamma(\Sigma)$ with
\begin{equation}\label{eq:K-S-int}
 \nabla_X \psi =cX\cdot\psi,
\end{equation}
for some constant $c$, where $\nabla$ is the Levi-Civita connection, $X$ any tangent vector, and $X\cdot\psi$ denotes
Clifford multiplication.  An easy computation shows that $\Ric_g =4(n-1)c^2g$.  Thus $c$ must be either purely imaginary in
which case $M$ is non-compact, $c=0$ with $\psi$ a parallel spinor and $M$ is Ricci-flat, or $c$ is real and
$M$ is positive Einstein and compact assuming completeness.  In the latter case $\psi$ is a \emph{real Killing spinor}.
We will only consider real Killing spinors with $c\neq 0$.  Since $c$ is rescaled by homotheties of the metric, only its
sign is of significance.  We denote by $N_{+}$ (respectively $N_{-}$) the dimension of the space of Killing spinors
with $c>0$ (respectively $c<0$).

Killing spinors are of interest in physics in supergravity and string theories~\cite{DufNilPop86}.  But they are also of interest
purely mathematically.  See~\cite{BFGK91} for a survey.  Much work has been done in classifying manifolds admitting a Killing
spinor.  C. B\"{a}r~\cite{Bar93} classified simply connected manifolds admitting a real Killing spinor in terms of the
underlying geometry of $(M,g)$.  The classification is given in terms of the holonomy of the metric cone
$(C(M),\ol{g}),\ C(M)=\R_{+} \times M,$ $\ol{g}=dr^2 +r^2 g$.  The argument in~\cite{Bar93} is essentially that
the connection $\nabla_X -cX$ on $\Sigma$ is identified with the Levi-Civita connection $\ol{\nabla}$ of $\ol{g}$
on $\ol{\Sigma}$ (the spin bundle of $C(M)$ when n is even, and half-spin bundle when n is odd).
Then the classification is in terms of irreducible holonomies admitting a parallel spinors~\cite{Wan89}.  See
Table~\ref{tab:Kill-spin} for the classification.  Therefore, just as for the irreducible reduced Ricci-flat holonomies there are two
cases occurring in infinitely many dimensions, the Sasaki-Einstein and 3-Sasakian manifolds, and two exceptional cases,
nearly K\"{a}hler and weak $G_2$ in dimensions 6 and 7 respectively.

Nearly K\"{a}hler structures, introduced by A. Gray in the context of weak holonomy,
are almost Hermitian structures $(g,J,\omega)$ with $\nabla_X J (X)=0$ for any $X\in TM$.  Note that for a proper nearly K\"{a}hler
structure, i.e. not K\"{a}hler, the almost complex structure $J$ is not integrable and $d\omega\neq 0$.  When $n=6$ the torsion of the
$\SU(3)$-structure is contained in a 1-dimensional subbundle.  In~\cite{Nag02} it is shown that every nearly K\"{a}hler
manifold is locally the Riemannian product of K\"{a}hler manifolds, 3-symmetric spaces, twistor spaces over positive
quaternion-K\"{a}hler manifolds and 6-dimensional nearly K\"{a}hler manifolds.  Thus most questions about nearly K\"{a}hler
manifolds reduces to proper 6-dimensional nearly K\"{a}hler manifolds.

A weak $G_2$ manifold is a 7-manifold with a vector cross product
coming from the imaginary octonians, or equivalently a \emph{stable} 3-form $\sigma\in\Omega^3$ with
$d\sigma =-\lambda\star\sigma$ with $\lambda\neq 0$ a constant.  The form $\sigma$ defines a reduction of the structure group of $M$ to
$G_2$ and thus a metric $g$, as $\G_2 \subset\SO(7)$, which is Einstein with scalar curvature $s=\frac{21}{8}\lambda^2$.  Again, the torsion of the $G$-structure lies in a 1-dimensional subbundle.
See~\cite{FKMS97} for results on weak $G_2$ manifolds including a classification of homogeneous examples.

Most interesting is perhaps $n=7$ for which, when $M$ is simply connected and not of constant curvature,
$N_{+}= 1,2,$ or 3, in which case $(M,g)$ is
said to be of type 1, 2, or 3 respectively.  Recall that the spinor representation $\mathbb{S}$ of $\Spin(7)$ is real,
$\mathbb{S}=\mathbb{S}_{\R} \otimes\C$.
Thus $M$ has a real spinor bundle $\Sigma_{\R}$, and the space of solutions to (\ref{eq:K-S-int}) is the complexification of
solutions in $\Gamma(\Sigma_{\R})$.  Each section $\psi\in\Gamma(\Sigma_{\R})$ defines a $G_2$-structure on $M$ with
stable 3-form $\sigma_{\psi}$, and there is a bijective correspondence between sections of $\mathbb{P}(\Sigma_{\R})$ and
$G_2$-structures with metric $g$ and given orientation.  If $\psi$ is a representative of such a section with $|\psi|=1$,
then $\sigma_{\psi}$ defines a weak $G_2$-structure, $d\sigma_{\psi}=-\lambda\star\sigma_\psi$, if and only if $\psi$ satisfies (\ref{eq:K-S-int}), with $\lambda=8c$.  If $(M,g)$ is type 1, then there is a unique 3-form inducing the given metric and
orientation.  If it is of type 2, then $(M,g)$ is Sasaki-Einstein but not 3-Sasakian and there is a space of compatible
3-forms parameterized by $\rps^1$.  And if it is of type 3, then $(M,g)$ is 3-Sasakian and has a space of compatible 3-forms
parameterized by $\rps^2$.  See~\cite{FKMS97}.

Note that an easy computation of the curvature of the warped product shows that $(C(M),\ol{g})$ is Ricci-flat
if and only if $(M,g)$ is Einstein with $\Ric_g =(n-1)g$.  Thus the classification as in Table~\ref{tab:Kill-spin} gives
a natural scaling in which $c=\pm\frac{1}{2}$ in (\ref{eq:K-S-int}) and $s=n(n-1)$.

\begin{table}
\caption{real Killing spinors}
\centering
\begin{tabular}{|c|c|c|c|c|}\hline
$\dim M$ & $N_+$ & $N_-$ & $\Hol(C(M))$ & geometry \\\hline\hline
$n$     & $2^{\lfloor\frac{n}{2}\rfloor}$ & $2^{\lfloor\frac{n}{2}\rfloor}$ & $Id$ & n-sphere \\\hline
$4m-1$  & 2     & 0     & $\SU(2m)$    &  Sasaki-Einstein \\\hline
$4m+1$  & 1     & 1     & $\SU(2m+1)$  &  Sasaki-Einstein \\\hline
$4m-1$  & m+1   & 0     & $\Sp(m)$     &  3-Sasakian  \\\hline
6       & 1     & 1     & $\G_2$       &  nearly K\"{a}hler \\\hline
7       & 1     & 0     & $\Spin(7)$   &  weak $\G_2$ \\\hline
\end{tabular}
\label{tab:Kill-spin}
\end{table}

We consider deformations of the Killing spinor equation (\ref{eq:K-S-int}) under deformations of $g$, both infinitesimal and
genuine.  As solutions to (\ref{eq:K-S-int}) imply that $(M,g)$ is Einstein we consider Einstein deformations.
The beginnings of a general theory of deformations of Killing spinors was developed by M. Wang~\cite{Wan91}, making use of
the work of J.-P. Bourguignon and P. Gauduchon~\cite{BouGau92} on the variations of spinors under metric variations.

More recently there has been some work on the two exceptional cases in Table~\ref{tab:Kill-spin}.  In~\cite{MorNagSem08} and~\cite{MorSem11} it is shown that the space of infinitesimal Einstein deformations of a proper nearly K\"{a}hler 6-manifold
consists of eigenspaces of the Laplace operator $\Delta$ restricted to the space $E$ of co-closed primitive $(1,1)$-forms.
If $E(\lambda)$ denotes the $\lambda$-eigenspace of $\Delta$ restricted to $E$, then the space of essential infinitesimal
Einstein deformations is $E(2)\oplus E(6)\oplus E(12)$.  The space of infinitesimal deformations of nearly K\"{a}hler structures
is $E(12)$.  Besides $S^6$, which has no Einstein deformations the only examples of proper nearly K\"{a}hler 6-manifolds are
3-symmetric spaces, $\cps^3 =\SO(5)/\U(2)$, $F(1,2)=\SU(3)/{\U(1)\times\U(1)}$, and
$S^3 \times S^3 =\SU(2)\times\SU(2)\times\SU(2)/\Delta$.  In~\cite{MorSem10} it is shown that the nearly K\"{a}hler structures on
$\cps^3$ and $S^3 \times S^3$ have no infinitesimal Einstein deformations, and on $F(1,2)$ $E(2)$ and $E(6)$ vanish while
$E(12)$ is an 8-dimensional space.

Similar results are known for weak $G_2$ manifolds.  In~\cite{AleSem11} a similar decomposition of the infinitesimal Einstein
deformations on a weak $G_2$ manifold are given.  First recall that a $G_2$-structure induces a decomposition of the 3-forms into
irreducible $G_2$-representations $\Lambda^3 =\Lambda^3_1 \oplus\Lambda^3_7 \oplus\Lambda^3_{27}$.  And there is a map
$\iota: S^2_0(T^*) \rightarrow\Lambda^3$, which on a decomposable element $\alpha\odot\beta$ is
$\iota(\alpha\odot\beta)= \alpha\wedge(\beta\contr\sigma) +\beta\wedge(\alpha\contr\sigma)$, which is an isomorphism onto
$\Lambda^3_{27}$.  It is proved in~\cite{AleSem11} that the essential infinitesimal Einstein deformations is given by the direct sum
\[ E(16)\oplus E(4)\oplus E(8),\]
where $E(16)=\{\gamma\in\Omega^3_{27} | \star d\gamma=-4\gamma\}$, $E(4)=\{\gamma\in\Omega^3_{27} | \star d\gamma=2\gamma\}$,
and $E(8)=\{\gamma\in\Omega^3_{27} | dd^*\gamma=8\gamma\}$.  The notation $E(\lambda)$ indicates that these are subspaces of
the $\lambda$-eigenspace of $\Delta$.  The space $E(16)$ is the subspace of infinitesimal deformations of weak $G_2$-structures,
or more precisely, those not fixing the metric and deforming the Killing spinor.
This space is computed on the normal homogeneous examples: the isotropy irreducible space $\SO(5)/\SO(3)$, the pinched metric on
$S^7$, and the second Einstein metric on the Aloff-Wallach space $N(1,1)=\SU(3)/\U(1)$.  The first two cases have no infinitesimal
Einstein deformations, while for the third the infinitesimal Einstein deformations correspond to $E(16)$ which is 8-dimensional.

These results might lead one to suspect that there might be some stability for Killing spinors under Einstein deformations, either
infinitesimal or integrable.  Furthermore, for the case $c=0$ in (\ref{eq:K-S-int}), i.e. parallel spinors, there are strong stability
results~\cite{Wan91,Nor10}.  Recall that a simply-connected, spin, irreducible Riemannian manifold $(M,g)$ admits a parallel
spinor if and only if the holonomy $\Hol(g)=G$ where $G=\SU(m), \Sp(m),\G_2,$ or $\Spin(7)$.  Define a $G$-manifold to be
a connected oriented manifold of dimension $2m, 4m, 7$ or $8$ respectively with a torsion-free $G$-structure with $G$ from this list.
This means $\Hol(g)\subseteq G$.  Thus a $G$-manifold $M$ is Ricci-flat, and we define $\mathcal{W}_G$ to be the moduli space
of torsion-free $G$-structures on $M$, $\mathcal{M}_G$ the moduli space of $G$-metrics, i.e. metrics induced by a torsion-free
$G$-structure, and $\mathcal{M}_0$ the moduli space of Ricci-flat metrics on $M$.  Here the moduli spaces are defined by quotienting
by diffeomorphisms isotopic to the identity.
We have the following result of J. Nordstr\"{o}m extending similar results of M. Wang~\cite{Wan91}.
\begin{thmint}[\cite{Nor10}]
Let $M$ be a compact $G$-manifold with $G=\SU(m), \Sp(m),\G_2,$ or $\Spin(7)$.  Then $\mathcal{M}_G$ is open in $\mathcal{M}_0$, actually a union of connected components.
Furthermore, $\mathcal{M}_G$ is a smooth manifold and the natural map
\[ m: \mathcal{W}_G \rightarrow \mathcal{M}_G \]
that sends a torsion-free $G$-structure to the metric it defines is a submersion.
\end{thmint}

This article will show that there is no analogous result for Killing spinors.  Under Einstein deformations
$N_{+}, N_{-}$ are merely upper semi-continuous and can drop under infinitesimal and integrable Einstein deformations.
In particular, the \emph{toric} 3-Sasakian 7-manifolds of~\cite{BGMR98} have interesting infinitesimal Einstein deformations.
Let $H^1(\mathcal{A}^\bullet)$ be the first cohomology of the complex (\ref{eq:Dol-comp}), that is the first order
deformations of the complex structure of the Reeb foliation $\mathscr{F}_{\xi}$.  We know that
$\dim_{\C}H^1(\mathcal{A}^\bullet)=b_2(M)-1$ if $(M,g)$ is a toric 3-Sasakian 7-manifold~\cite{vanCo06}.

\begin{thmint}\label{thmint: inf-Einst-tor-3S}
Let $(M,g)$ be a 3-Sasakian 7-manifold with $\dim_{\C}H^1(\mathcal{A}^\bullet)>0$, e.g. a toric 3-Sasakian 7-manifold with
$b_2(M)\geq 2$.  Thus $(M,g)$ has three linearly independent Killing spinors.
Then there exist infinitesimal Einstein deformations of $g$ preserving two, one, and zero dimensional subspaces of
the Killing spinors.
\end{thmint}

It is unknown whether the infinitesimal Einstein deformations preserving only 1-dimensional subspaces of Killing spinors or none
are integrable.  But in Section~\ref{sec:int-def} some infinitesimal Einstein deformations are proved to be
integrable.  For example the infinitesimal deformations of a toric 3-Sasakian 7-manifold in the theorem preserving a
2-dimensional subspace of Killing spinors can be shown to be integrable.
\begin{thmint}\label{thmint:Einst-def-tor-3S}
Let $(M,g)$ be a \emph{toric} 3-Sasakian 7-manifold, so $N_{+}=3$.  There exists an effective space
$\mathcal{U}\subset\C^{b_2(M)-1}$ of Einstein deformations of $g=g_0$. For $t\in\mathcal{U}$ and $t\neq 0$,
$g_t$ is Sasaki-Einstein but not 3-Sasakian. Thus $g_t, t\neq 0,$ admits only a two dimensional space of Killing spinors
($N_{+} =2, N_{-} =0$).
\end{thmint}

We also prove in Theorem~\ref{thm:3Sasak-int-def} that certain infinitesimal Einstein deformations on a general 3-Sasakian manifold
are integrable.  In Section~\ref{subsec:3Sasak-def} we see that this has implications for the local premoduli space of Einstein
metrics.
\begin{corint}\label{corint:Einst-def-3S}
Suppose $(M,g)$ is 3-Sasakian with $\dim_{\C} H^1(\mathcal{A}^\bullet) >0$, e.g. a toric 3-Sasakian 7-manifold with $b_2(M)\geq 2$.
Then either there exist Einstein deformations of $g$ preserving no Killing spinors, or the Einstein premoduli space is singular.
\end{corint}

In Section~\ref{sec:prelim} we review necessary background on the deformations of Einstein metrics, the variation of
spin structures, and deformations of Killing spinors.  In Section~\ref{sec:S-E-KS-def} we show that infinitesimal deformations
of the transversal complex structure of a Sasaki-Einstein manifold give infinitesimal Einstein deformations.  We then give
the basic results on these deformations regarding the behavior of Killing spinors, on Sasaki-Einstein and 3-Sasakian manifolds.
In Section~\ref{sec:int-def} we give some results on when these infinitesimal Einstein deformations integrate to genuine Einstein deformations.  In Section~\ref{subsec:3Sasak-def} we study the space of these infinitesimal Einstein deformations on a 3-Sasakian manifold more closely, and we prove Theorem~\ref{thmint: inf-Einst-tor-3S}, Theorem~\ref{thmint:Einst-def-tor-3S}
and Corollary~\ref{corint:Einst-def-3S}.  In Section~\ref{subsec:examples} the examples of toric 3-Sasakian 7-manifolds from~\cite{BGMR98} provide non-trivial examples of the above results.

\subsection*{Acknowledgements}

I would like to thank the Max Planck Institute for Mathematics for their hospitality and excellent research environment.
Most of the research for this article was done during a visit during the academic year 2011-2012.

\section{Preliminaries}\label{sec:prelim}

\subsection{Spinors}\label{subsec:spinors}

We review the explicit construction of the spin representations via explicit representations of the Clifford algebras $\Cl(n)$,
For more details see~\cite{LawMi89} and~\cite{BFGK91}.
These representation will give the complex representations of the complex Clifford algebras $\CCl(n)=\Cl(n)\otimes\C$.
Suppose $V$ is a real vector space of dimension $n=2m$ with a metric $g$ and compatible almost complex structure $I:V\rightarrow V$.
We have the decomposition $V\otimes\C = V^{1,0}\oplus V^{0,1}$, and the spinor space is
\[ \mathbb{S}(V):=\Lambda ^{*,0} V =\Lambda^* V^{1,0}.\]
The representation $c:\CCl(V) \rightarrow\End(\mathbb{S}(V))$ is defined by its action on $V\otimes\C$.
For $v\in V^{1,0}$ define $c(v):=\sqrt{2}v\wedge\cdot$, and for $w\in V^{0,1}$ define $c(w) :=-\sqrt{2}w\contr\cdot$, where the contraction is induced by the metric $g$ on $V$ extended complex bilinearly.

Recall we have the splitting $\Cl(V)=\Cl_0(V)\oplus\Cl_1(V)$ into even and odd elements making $\Cl(V)$ into a \emph{superalgebra},
that is
\[ \Cl_r(V)\cdot\Cl_s(V)\subseteq\Cl_t(V)\text{  with  }  t=r+s \mod 2.\]
We have $\Pin(n)\subset\Cl(n)$, where $\Pin(n)$ is the universal cover of $\Or(n)$, and $\Spin(n)\subset\Cl_0(n)$ is the
universal cover of $\SO(n)$.

The representation has a splitting preserved by the superalgebra structure of $\CCl(V)$
\begin{equation}\label{eq:Spin-rep}
\mathbb{S}(V)=\mathbb{S}_{2m} =\mathbb{S}_{2m}^+ \oplus\mathbb{S}_{2m}^-,
\end{equation}
that is $\CCl_0(V)\cdot\mathbb{S}_{2m}^\pm \subseteq \mathbb{S}_{2m}^\pm$ while
$\CCl_1(V)\cdot\mathbb{S}_{2m}^\pm \subseteq \mathbb{S}_{2m}^\mp$.
The restriction of $\mathbb{S}(V)$ to $\Spin(2m)$ is the \emph{spin representation}, which splits into components
in (\ref{eq:Spin-rep}) which are irreducible.

As in~\cite{Wan89}, we define $\mathbb{S}_{2m}^+$ to be the half-spin representation with highest weight
$\frac{1}{2}(x_1 +\cdots +x_m)$, while $\mathbb{S}_{2m}^-$ has highest weight $\frac{1}{2}(x_1 +\cdots +x_{m-1} -x_m)$, with
the usual choice of fundamental weights.
If $\{e_1,\ldots, e_{2m}\}$ is an orthonormal basis of $V$, then $\mathbb{S}_{2m}^\pm$ are the $+1$ and $-1$ eigenspaces of
$\omega_{\C} =\bigl(\sqrt{-1}\bigr)^{m^2 +2m}e_1 \cdots e_{2m}$.

\begin{remark}
Note that this differs from the convention in~\cite{LawMi89}, where $\mathbb{S}_{2m}^\pm$ are defined as the
$+1$ and $-1$ eigenspaces of $\omega_{\C} =\bigl(\sqrt{-1}\bigr)^{m}e_1 \cdots e_{2m}$,
by a factor of $(-1)^{\frac{m(m+1)}{2}}$.
\end{remark}

Explicitly, we have
\begin{equation}
\begin{split}
\mathbb{S}_{2m}^+ & = \Lambda^{m,0}V \oplus\Lambda^{m-2,0}V\oplus\cdots ,\\
\mathbb{S}_{2m}^- & = \Lambda^{m-1,0}V\oplus\Lambda^{m-3,0}V\oplus\cdots.
\end{split}
\end{equation}

For the odd dimensional case, $n=2m+1$, let $\{e_1,\ldots, e_{2m}\}$ be an orthonormal basis of $V$ and define
$V' =V\oplus\R e_{2m+1}$, with
$e_{2m+1}$ unit length and orthogonal to $V$.  We define $c':\Cl(V')\rightarrow\End(\mathbf{S}(V))$ as follows.
If $v\in V$ we let $c'(v):=c(v)\in\End(\mathbf{S}(V))$ as above, and we define $c'(e_{2m+1}):=-(-1)^{\frac{m+1}{2}}c(e_1 \cdots e_{2m})\in\End(\mathbf{S}(V)).$
Note that $\CCl(V')=\CCl(2m+1)$ has two irreducible complex representations, each of dimension $2^m$, and changing the sign of $c'(e_{2m+1})$ gives the other representation of $\CCl(V')$.

Alternatively, let $V=V_0 \oplus\R e_{2m}$ be an orthogonal sum.  Then
\begin{equation}\label{eq:odd-rep}
\Cl(V_0)\overset{\gamma}{\cong}\Cl_0 (V)\overset{c}{\longrightarrow}\End(\mathbb{S}^{\pm}(V)),
\end{equation}
where the isomorphism $\gamma:\Cl(V_0)\overset{\gamma}{\cong}\Cl_0(V)$ is given by $e_i \mapsto e_i \cdot e_{2m}$.  The choice of
half-spin representations $\mathbb{S}^{\pm}(V)$ gives the two representations of $\CCl(V_0)$ denoted by
$\mathbb{S}^{\pm}_{2m-1}$.  The restrictions of $\mathbb{S}^{\pm}_{2m-1}$ to $\CCl_0(V_0)$ are identical, thus
restricting to $\Spin(2m-1)\subset\CCl_0(V_0)$ gives the complex spin representation $\mathbb{S}_{2m-1}$, without a superscript.

Let $(M,g)$ be an oriented Riemannian manifold with a spin structure.  We have the principal bundle of orthonormal frames
$L_{\SO(n)}$ with the spin structure a $\Spin(n)$ principal bundle $L_{\Spin(n)}$ with 2-fold cover
$\theta:L_{\Spin(n)} \rightarrow L_{\SO(n)}$, restricting to the 2-fold cover $\Spin(n)\rightarrow\SO(n)$ on each fiber.
The \emph{spin bundle} is $\Sigma =L_{\Spin(n)} \times_{\Spin(n)} \mathbb{S}_{n}$.  If $n=2m$ then
$\Sigma =\Sigma^+ \oplus\Sigma^-$, where $\Sigma^{\pm} = L_{\Spin(n)} \times_{\Spin(n)} \mathbb{S}^{\pm}_{n}$.
When $n$ is odd there is a unique spinor bundle $\Sigma$, although there are two choices as a bundle of Clifford modules
over $\CCl(TM)$.

Since Killing spinors correspond to a holonomy reduction we will make use of the decomposition of some restrictions of the
spinor representation $\mathbb{S}_n$.  Let $\mu_m$ be the usual representation of $\SU(m)\subset\SO(2m)$ on $\C^m$.
Since $\SU(m)$ is simply connected, $\SU(m)\subset\SO(2m)$ lifts to an embedding $\SU(m)\subset\Spin(2m)$ under
$\theta:\Spin(2m)\rightarrow\SO(2m)$.  We have from our conventions
\begin{equation}\label{eq:spin-rep-su}
\begin{aligned}
\mathbb{S}^+_{2m} |_{\SU(m)} & = \Lambda^m \mu_m \oplus\Lambda^{m-2} \mu_m \oplus\cdots \\
\mathbb{S}^-_{2m} |_{\SU(m)} & = \Lambda^{m-1}\mu_m \oplus\Lambda^{m-3} \mu_m \oplus\cdots
\end{aligned}
\end{equation}

We will need to consider the spin representation restricted to
$\LSp(m)\oplus\LSp(1)\subset\SO(4m)$.  Let $\nu_{2m}$ be the complex representation of $\Sp(m)$ given by
$\Sp(m)\subset\SU(2m)$.  Contraction by the symplectic form gives $\Lambda^k \nu_{2m} =\Lambda_k \oplus\Lambda^{k-2}\nu_{2m}$,
for $2\geq k\geq m$, as $\Sp(m)$-representations where $\Lambda_k$ is the irreducible representation of $\Sp(m)$ with
highest weight $x_1 +\cdots +x_k$.  It is an elementary result (see~\cite[Prop. 4.14]{BroDie85})that an irreducible representation of $\Sp(m)\times\Sp(1)$ is of the form $V\hat{\otimes}W$ where $V$ and $W$ are irreducible representations of $\Sp(m)$ and
$\Sp(1)$ respectively.  A little more work shows that
\begin{equation}\label{eq:spin-rep-quat}
\begin{aligned}
\mathbb{S}^+_{4m} |_{\LSp(m)\oplus\LSp(1)} & = \Lambda_0 \hat{\otimes}\gamma_m \oplus\Lambda_2 \hat{\otimes}\gamma_{m-2}\oplus\cdots \\
\mathbb{S}^-_{4m} |_{\LSp(m)\oplus\LSp(1)} & = \Lambda_1 \hat{\otimes}\gamma_{m-1} \oplus\Lambda_3 \hat{\otimes}\gamma_{m-3}\oplus\cdots
\end{aligned}
\end{equation}
where $\gamma_k =S^k(\mu_2)$ is the irreducible representation of $\SU(2)=\Sp(1)$ of dimension $k+1$.
It follows from (\ref{eq:spin-rep-quat}) that for m even the inclusion
$\Sp(m)\cdot\Sp(1)=\Sp(m)\times\Sp(1)/\Z_2 \subset\SO(4m)$ lifts under $\theta:\Spin(4m)\rightarrow\SO(4m)$ to
$\Sp(m)\times\Sp(1)/\Z_2 \subset\Spin(4m)$.  While when m is odd
$\theta^{-1}(\Sp(m)\cdot\Sp(1)) =\Sp(m)\times\Sp(1)\subset\Spin(4m)$, which contains $(-I,-1)=-1\in\Spin(4m)$.

\subsection{Deformation of Einstein metrics and Killing spinors}

\subsubsection{Deformation of Einstein metrics}\label{subsec:Einst-def}
We describe what we will need from the theory of deformations of Einstein metrics and deformations of Killing spinors.
For more on the deformation theory of Einstein metrics see~\cite[ch. 12]{Bes87} or~\cite{Koi83}.  See~\cite{BouGau92} for the
apparatus for working with spinors under metric variations, and see~\cite{Wan91} for this applied to the Killing spinor equation.
In this article $M$ denotes a compact connected n-dimensional manifold.

\begin{defn}
Let $g$ be an Einstein metric on $M$.  A family $g_t$ of Einstein metrics on $M$ of fixed volume with $g_0 =g$ depending
smoothly on $t\in U\subset\R^k$ is an \emph{Einstein deformation} of $g$.
\end{defn}
Because Einstein metrics are critical points of the total scalar curvature functional $g\mapsto\int_M s_g\,\mu_g$ restricted to metrics of a fixed volume, a deformation of Einstein metrics has fixed scalar curvature $s =s_{g_t}$.  Thus
\begin{equation}\label{eq:Einst}
\Ric_{g_t} =\lambda g_t,
\end{equation}
where $\lambda =\frac{s}{n}$.  We will consider positive scalar curvature Einstein metrics, and it will be convenient for us to assume $\lambda =n-1$.

Let $\mathscr{M}_c$ be the space of Riemannian metrics on $M$ of fixed volume $c$.  This is acted upon by the diffeomorphism group
$\mathscr{D}$.  A local description of the quotient $\mathscr{M}_c /\mathscr{D}$ is given by D. Ebin's Slice Theorem~\cite{Ebi70}.
The tangent space to $\mathscr{M}_c$ at $g$ denoted by $T_g \mathscr{M}_c$ consists of
symmetric 2 tensors $h\in\Gamma\bigl(\Sym^2 T^*M \bigr)$ with $\int_M \tr h\, \mu_g =0$.  The tangent space to the orbit
$\mathscr{D}^*g$ consists of all Lie derivatives $\mathcal{L}_X g =2\delta^*_g X^{\flat}$, where $X^{\flat}$ is the 1-form dual to
a the vector field $X$ and
\begin{equation}
(\delta_g^*X^{\flat})_{ij}  =\frac{1}{2}\bigl(\nabla_i X^{\flat}_j +\nabla_j X^{\flat}_i \bigr),
\end{equation}
with $\nabla$ the Levi-Civita connection.  One can show that $\im\delta_g^* \subset T_g \mathscr{M}_c$ is closed, and
\begin{equation}\label{eq:slice-inf}
T_g \mathscr{M}_c =\im\delta_g^* \oplus\bigl(T_g \mathscr{M}_c \cap\ker\delta \bigr),
\end{equation}
where $(\delta_g h)_i =-\nabla^j h_{ji}$ is adjoint to $\delta_g^*$.

Let $h=\frac{dg_t}{dt}|_{t=0}$, then differentiating (\ref{eq:Einst}) gives
\begin{equation}\label{eq:Einst-lin}
2E'_g(h) = \bigl(\ol{\Delta} +2L-\delta_g^*\delta_g-\nabla d\tr_g\bigr)h =0,
\end{equation}
where $(Lh)_{ij} =R\indices{_i^k_j^l}h_{kl}$ and $\ol{\Delta}=\nabla^* \nabla$ is the rough Laplacian.

\begin{defn}
Let $(M,g)$ be an Einstein manifold.  A symmetric 2-tensor $h\in\Gamma\bigl(\Sym^2 T^*M \bigr)$ is an
\emph{infinitesimal Einstein deformation} of $g$ if $h$ satisfies (\ref{eq:Einst-lin}) and $\int_M \tr_g h\, \mu_g =0$.  The space of infinitesimal Einstein deformations is denoted by $\ED(g)$.
\end{defn}

An infinitesimal Einstein deformations of the form $\mathcal{L}_X g$ is said to be \emph{trivial}.  The space of trivial
infinitesimal Einstein deformations is denoted by $\TED(g)$.  An infinitesimal Einstein deformation $h$ is said to be
\emph{essential} if it is orthogonal to $\TED(g)$.  The space of essential infinitesimal Einstein deformations is
denoted by $\EED(g)$.  We can use the following lemma due to M. Berger and D.G. Ebin as the definition of $\EED(g)$.
\begin{lem}[\cite{BerEbe69}]
Let $(M,g)$ be an Einstein manifold.  An $h\in\Gamma\bigl(\Sym^2 T^*M \bigr)$ is an element of $\EED(g)$ if and only if $h$ satisfies
\begin{equation}
\bigl(\ol{\Delta} +2L \bigr) h=0,\quad \delta_g h =0,\quad \tr_g h=0.
\end{equation}
\end{lem}

We have the decomposition of closed spaces
\begin{equation}
\ED(g) =\EED(g)\oplus \TED(g),
\end{equation}
with $\EED(g)$ finite dimensional.

\begin{defn}
Let $(M,g)$ be an Einstein manifold.  The subset of Einstein metrics in the Ebin slice $\mathscr{S}_g$ (cf.~\cite{Ebi70})
at $g$ is called the \emph{local premoduli space of Einstein structures} and denoted by $\PM(g)$.
\end{defn}
The local moduli space is $\PM(g)/\Isom(g)$, but it will be more convenient to work with the local premoduli space.

\subsubsection{Deformation of spinors}\label{subsubsec:spin-def}

We will need the machinery due to J.P. Bourguignon and P. Gauduchon~\cite{BouGau92} for describing variations of spinor bundles and
spinors under metric variations and applied by M. Wang~\cite{Wan91} to study Killing spinor variations.

Let $P=L_{\SO(n)}$ be the bundle of oriented orthonormal frames on $(M,g)$.  A spin structure is a double cover $\tilde{P}$.
Given a symmetric, with respect to $g$, automorphism $\alpha:TM\rightarrow TM$ we have a new metric
\[g^\alpha (X,Y) =g(\alpha^{-1} X,\alpha^{-1}Y).\]
If $P^\alpha$ is the bundle of $g^\alpha$-orthonormal oriented frames, $\alpha :P\rightarrow P^\alpha$ is $\SO(n)$-equivariant,
and gives an isomorphism
\[ \Sigma =\tilde{P} \times_{\Spin(n)} \mathbb{S}_n \overset{\tilde{\alpha}}{\rightarrow} \Sigma^\alpha =\tilde{P}^\alpha \times_{\Spin(n)} \mathbb{S}_n . \]

Let $\alpha(t)$ be a smooth path of symmetric automorphisms with $\alpha(0)=\mathbb{1}_{TM}$, and
$\hat{\sigma}_t$ Killing spinors for $g^{\alpha(t)}$,
\begin{equation*}
\nabla^{\alpha(t)}_X \hat{\sigma}_t =cX\cdot_t \hat{\sigma}_t.
\end{equation*}

Set $\sigma_t =\tilde{\alpha}(t)^{-1}(\hat{\sigma}_t)$, then in terms of the original spin bundle
\begin{equation}\label{eq:Kill-spin-def1}
\ol{\nabla}^{\alpha(t)}_X \sigma_t =c\alpha(t)^{-1}(X)\cdot\sigma_t ,
\end{equation}
where $\ol{\nabla}^{\alpha(t)}_X =\tilde{\alpha}(t)^{-1}\circ\nabla^{\alpha(t)}_X \circ\tilde{\alpha}(t)$.

A \emph{deformation of the Killing spinor} $\sigma_0$ is a path $(\alpha(t),\sigma_t)$ satisfying
\begin{equation}\label{eq:Kill-spin-def-L}
\mathcal{L}^c(\alpha(t),\sigma_t)(X):=\ol{\nabla}^{\alpha(t)}_X \sigma_t -c\alpha(t)^{-1}(X)\cdot\sigma_t =0.
\end{equation}

We will make use of the twisted Dirac operator
\begin{equation}\label{eq:Dirac-twist}
\mathcal{D}: \Gamma(TM^*_{\C} \otimes\Sigma) \rightarrow \Gamma(TM^*_{\C} \otimes\Sigma).
\end{equation}
Decomposing into irreducible representations of $\Spin(n)$
\[ TM^*_{\C} \otimes\Sigma =\Sigma\oplus\Sigma_{\frac{3}{2}},\]
where $\Sigma_{\frac{3}{2}}$ is the bundle coming from the kernel of Clifford multiplication
$p: T\otimes\mathbb{S}_n \rightarrow\mathbb{S}_n$.
The component of $\mathcal{D}$ on $\Sigma_{\frac{3}{2}}$ is the \emph{Rarita-Schwinger operator}
\begin{equation}
\mathcal{Q}: \Gamma(\Sigma_{\frac{3}{2}}) \rightarrow \Gamma(\Sigma_{\frac{3}{2}}).
\end{equation}
If $\Psi\in\Gamma(\Sigma_{\frac{3}{2}})$ then $\mathcal{D}\Psi =\mathcal{Q}\Psi$ if and only if $\delta_g \Psi=0$.

We define tensors $\Psi^{(\beta,\sigma_0)}, \Theta^{(\beta,\sigma_0)}\in\Gamma(T^*M_{\C} \otimes\Sigma)$ for
$\beta:TM\rightarrow TM$ and $\sigma_0 \in\Gamma(\Sigma)$:
\begin{align}\label{eq:def-psi}
\Psi^{(\beta,\sigma_0)}(X) & =\beta(X)\cdot\sigma_0 \\\label{eq:def-theta}
\Theta^{(\beta,\sigma_0)}(X) & =\sum_i e_i(\nabla_{i} \beta)(X)\cdot\sigma_0,
\end{align}
where $X\in TM$ and $\{ e_i\}$ is a local orthonormal frame.  If $\beta$ is symmetric, $\tr_g \beta =0$,
and $\delta_g \beta =0$ then $\Psi^{(\beta,\sigma_0)}, \Theta^{(\beta,\sigma_0)}\in\Gamma(\Sigma_{\frac{3}{2}})$.
And if $\sigma_0$ is a Killing spinor, then $\delta_g \Psi^{(\beta,\sigma_0)}=\delta_g \Theta^{(\beta,\sigma_0)} =0$.

Differentiating (\ref{eq:Kill-spin-def-L}) at $(\mathbb{1}_{TM},\sigma_0)$:
\begin{prop}[\cite{Wan91}]\label{prop:Kill-spin-inf}
\[d\mathcal{L}^c(\dot{\alpha},\dot{\sigma})(X)=\nabla\dot{\sigma}_X -cX\dot{\sigma} +c\dot{\alpha}(X)\sigma_0 -\frac{1}{2}\sum_i e_i(\nabla_i\dot{\alpha})(X)\sigma_0 +\frac{1}{2}g(\delta\dot{\alpha},X)\sigma_0.\]
If $\tr_g(\dot{\alpha})=\delta\dot{\alpha}=0$, then $d\mathcal{L}^c(\dot{\alpha},\dot{\sigma})=0$ if and only if
$\nabla_X \dot{\sigma} =cX\dot{\sigma}$ and $\mathcal{D}\Psi^{(\dot{\alpha},\sigma_0)}=nc\Psi^{(\dot{\alpha},\sigma_0)}$.
\end{prop}

For $\beta:TM\rightarrow TM$ $g$-symmetric, define $h(X,Y)=-2g(\beta(X),Y)$.
\begin{prop}[\cite{Wan91}]
If $\tr_g \beta =\delta\beta=0$ and $\mathcal{D}\Psi^{(\beta,\sigma_0)}=cn\Psi^{(\beta,\sigma_0)}$, then
$\bigl(\ol{\Delta} +2L \bigr) h=0$ where $(Lh)_{ij} =R\indices{_i^k_j^l}h_{kl}$.
\end{prop}
So $h\in\Gamma\bigl(\Sym^2 T^*M \bigr)$ is an infinitesimal Einstein deformation.

\begin{defn}
An \emph{infinitesimal deformation of the Killing spinor} $\sigma_0$ is a pair $(\beta,\sigma)$, $\beta:TM\rightarrow TM$
symmetric and $\sigma\in\Gamma(\Sigma)$, satisfying:
\begin{thmlist}
\item  $\sigma$ is a Killing spinor with constant $c$,

\item  $\tr_g \beta =\delta\beta=0$,

\item  $\mathcal{D}\Psi^{(\beta,\sigma_0)}=nc\Psi^{(\beta,\sigma_0)}$.
\end{thmlist}
\end{defn}

The following result will have applications for the existence of eigenvectors of $\mathcal{Q}$.
\begin{prop}[\cite{Wan91}]
Let $(M,g)$ be spin with nonzero Killing spinor $\sigma_0$.  Let $h\in\EED(g)$, and define $\beta:TM\rightarrow TM$
by $h(X,Y)=-2g(\beta(X),Y)$.  Then we have an eigenvector of $\mathcal{Q}$ of either eigenvalue $cn$ or $c(2-n)$, that is
\begin{thmlist}
\item  $\mathcal{D}\Psi^{(\beta,\sigma_0)}=nc\Psi^{(\beta,\sigma_0)}$ and $\beta$ is an infinitesimal deformation of $\sigma_0$, or

\item  $\Theta^{(\beta,\sigma_0)} -2c\Psi^{(\beta,\sigma_0)}\neq 0$ and
\[ \mathcal{D}\bigl(\Theta^{(\beta,\sigma_0)} -2c\Psi^{(\beta,\sigma_0)} \bigr)=c(2-n)\bigl(\Theta^{(\beta,\sigma_0)} -2c\Psi^{(\beta,\sigma_0)} \bigr).  \]
\end{thmlist}
\end{prop}

Let $(M,g)$ be Einstein, then the Einstein premoduli space $\PM(g)\subseteq Z$, where $Z$ is a finite dimensional
real analytic submanifold of the slice $\mathscr{S}_g$~\cite{Koi80}.
The bundles $\Sigma_{g'}$ and equation (\ref{eq:K-S-int}) depend real analytically on $g' \in Z$.  Define
$\mathcal{N}_{g'}^{+}$ (resp. $\mathcal{N}_{g'}^{-}$) to be space of solutions of (\ref{eq:K-S-int}) for $g' \in Z$
and $c=\frac{1}{2}$ (resp. $c=-\frac{1}{2}$).  Since (\ref{eq:K-S-int}) has injective symbol $\dim_\C \mathcal{N}_{g'}^{\pm}$
is upper semi-continuous.  See for example~\cite[Lemma 4.3]{Koi83}.  We will see by example that it is not locally constant
as in the case of parallel spinors.

\subsection{Sasakian manifolds}

\subsubsection{Sasakian structures}
The Killing spinor deformations we consider are of the non-exeptional cases of Sasakian and 3-Sasakian manifolds in Table~\ref{tab:Kill-spin}.  See~\cite{BoyGal99} or the monograph~\cite{BoyGal08} for more on Sasakian geometry.

\begin{defn}\label{defn:Sasaki}
A Riemannian manifold $(M,g)$ is Sasakian if the metric cone $(C(M),\ol{g})$, $C(M):=\R_+ \times M$ and $\ol{g}=dr^2 +r^2 g$, is K\"{a}hler, that is $\ol{g}$ admits a compatible almost complex structure $J$ so that $(C(M),\ol{g}, J)$ is a K\"{a}hler structure.  Equivalently, $\Hol(C(M),\ol{g})\subseteq\U(m)$, where $\dim M =n=2m-1$.
\end{defn}

It is convenient to identify $M$ with $\{r=1\}= \{1\}\times M \subset C(M)$.
A Sasaki structure is a special type of metric contact structure.
Traditionally the Sasakian structure on $M$ was defined as a metric contact structure $(g,\eta,\xi,\Phi)$ satisfying an additional condition called \emph{normality}, which is an integrability condition, where $\eta$ is a contact form with Reeb vector field $\xi$ and $\Phi$ is a $(1,1)$ tensor.  Here $\xi$ and $\eta$ are restrictions to $M$ of
\begin{equation}
\xi =J r\partial r,\quad \eta(X)=\frac{1}{r^2}\xi\contr\ol{g},
\end{equation}
on $C(M)$, which are given the same notation.  It follows from the latter formula that
\begin{equation}\label{eq:cont}
\eta=d^c \log r,
\end{equation}
where $d^c =\sqrt{-1}(\ol{\partial}-\partial)$.  One can show from the warped product structure of $(C(M),\ol{g})$ that
$\xi$ is Killing and real holomorphic.  If $\omega$ is the K\"{a}hler form of $\ol{g}$, then
\[ \omega =\frac{1}{2}d(r^2 \eta)=\frac{1}{4}dd^c r^2.\]
we also have
\begin{equation}\label{eq:Kah-con}
\omega=\frac{1}{2}d(r^2\eta) =rdr\wedge\eta +\frac{1}{2}r^2 d\eta.
\end{equation}

Let $D\subset TM$ be the contact distribution which is defined by
\begin{equation}
D_x =\ker\eta_x
\end{equation}
for $x\in M$.  There is a splitting of the tangent bundle $TM$
\begin{equation}
TM =D\oplus L_{\xi},
\end{equation}
where $L_{\xi}$ is the trivial subbundle generated by $\xi$.  The tensor $\Phi\in\End(TM)$ is defined by
$\Phi|_D =J$ and $\Phi(\xi) =0$.  Since $\xi$ is Killing one can show that $\Phi =\nabla\xi$.
We denote the Sasakian structure by $(g,\eta,\xi,\Phi)$.

The vector field $\xi +\sqrt{-1}r\partial_r$ is holomorphic on $C(M)$, thus it defines a holomorphic action of $\tilde{\C}^*$,
the universal cover of $\C^*$.  The intersection of each orbit with $M\subset C(M)$ is an orbit of the action of $\xi$ on
$M$.  Thus the orbits define a transversely holomorphic foliation $\mathscr{F}_\xi$ on $M$ called the \emph{Reeb foliation}.
If $\xi$ generates a free $\U(1)$-action, then the Sasakian structure is \emph{regular}.  The Sasakian structure is
\emph{quasi-regular} if it generates a locally free $\U(1)$-action, and \emph{irregular} if not all the orbits are compact.

The foliation $\mathscr{F}_{\xi}$ together with its transverse holomorphic structure is given by an open covering
$\{U_\alpha \}_{\alpha\in A}$ and submersions $\pi_\alpha :U_\alpha \rightarrow W_\alpha \subset\C^{m-1}$ such that
when $U_\alpha \cap U_\beta \neq\emptyset$ the map
\[\phi_{\beta\alpha} =\pi_\beta \circ\pi_\alpha^{-1} :\pi_{\alpha}(U_\alpha \cap U_\beta) \rightarrow\pi_{\beta}(U_\alpha \cap U_\beta) \]
is a biholomorphism.

Note that on $U_\alpha$ the differential $d\pi_\alpha :D_x \rightarrow T_{\pi_\alpha(x)}W_\alpha$ at $x\in U_\alpha$ is
an isomorphism taking the almost complex structure $J_x$ to that on $T_{\pi_\alpha(x)}W_\alpha$.
Since $\xi\contr d\eta =0$ the 2-form $\frac{1}{2}d\eta$ descends to a form $\omega_\alpha^T$ on $W_\alpha$.  Similarly,
$g^T =\frac{1}{2}d\eta(\cdot,\Phi\cdot)$ satisfies $\mathcal{L}_\xi g^T =0$ and vanishes on vectors tangent to the leaves, so
it descends to an Hermitian metric $g^T_\alpha$ on $W_\alpha$ with K\"{a}hler form $\omega_\alpha^T$.  The K\"{a}hler metrics
$\{g_\alpha ^T \}$ and K\"{a}hler forms $\{\omega_\alpha^T \}$ on $\{ W_\alpha\}$ by construction are isomorphic on the overlaps
\[ \phi_{\beta\alpha} : \pi_{\alpha}(U_\alpha \cap U_\beta) \rightarrow\pi_{\beta}(U_\alpha \cap U_\beta).\]
We will use $g^T$, respectively $\omega^T$, to denote both the K\"{a}hler metric, respectively K\"{a}hler form, on the the
local charts and the globally defined pull-back on $M$.

If we define $\nu(\mathscr{F}_\xi) =TM/{L_\xi}$ to be the normal bundle to the leaves, then we can generalize the above concept.
\begin{defn}
A tensor $\Psi\in\Gamma\bigl((\nu(\mathscr{F}_\xi)^*)^{\otimes p} \bigotimes\nu(\mathscr{F}_\xi)^{\otimes q}\bigr)$ is \emph{basic}
if $\mathcal{L}_V \Psi =0$ for any vector field $V\in\Gamma(L_\xi)$.
\end{defn}
  Note that it is sufficient to check the above property for $V=\xi$.
Then $g^T$ and $\omega^T$ are such tensors on $\nu(\mathscr{F}_\xi)$.  We will also make use of the bundle isomorphism
$\pi:D \rightarrow\nu(\mathscr{F}_\xi)$, which induces an almost complex structure $\ol{J}$ on $\nu(\mathscr{F}_\xi)$ so that
$(D,J)\cong(\nu(\mathscr{F}_\xi),\ol{J})$ as complex vector bundles.  Clearly, $\ol{J}$ is basic and is mapped to the
natural almost complex structure on $W_\alpha$ by the local chart $d\pi_\alpha :D_x \rightarrow T_{\pi_\alpha(x)}W_\alpha$.

To work on the K\"{a}hler leaf space we define the Levi-Civita connection of $g^T$ by
\begin{equation}\label{eq:LC-trans}
\nabla^T_X Y =\begin{cases}
\pi_\xi(\nabla_X Y) & \text{ if }X, Y\text{ are smooth sections of }D, \\
\pi_\xi([V,Y]) & \text{ if } X=V\text{ is a smooth section of }L_\xi,
\end{cases}
\end{equation}
where $\pi_\xi :TM \rightarrow D$ is the orthogonal projection onto $D$.  Then $\nabla^T$ is the unique torsion free connection
on $D\cong\nu(\mathscr{F}_\xi)$ so that $\nabla^T g^T=0$.  Then for $X,Y\in\Gamma(TM)$ and $Z\in\Gamma(D)$ we have the
curvature of the transverse K\"{a}hler structure
\begin{equation}\label{eq:trans-curv}
R^T(X,Y)Z =\nabla^T_X \nabla^T_Y Z -\nabla^T_Y \nabla^T_X Z -\nabla^T_{[X,Y]} Z,
\end{equation}
and similarly we have the transverse Ricci curvature $\Ric^T$ and scalar curvature $s^T$.  We will denote the
transverse Ricci form by $\rho^T$.
From O'Neill's tensors computation for Riemannian submersions \cite{ONe66} and elementary properties of Sasakian structures
we have the following.
\begin{prop}\label{prop:Sasaki-Ric}
Let $(M,g,\eta,\xi,\Phi)$ be a Sasakian manifold of dimension $n=2m-1$, then
\begin{thmlist}
\item  $\Ric_g (X,\xi) =(2m-2)\eta(X),\quad\text{for }X\in\Gamma(TM)$,\label{eq:submer-Ric-Reeb}
\item  $\Ric^T (X,Y) =\Ric_g (X,Y) +2g^T(X,Y),\quad\text{for }X,Y\in\Gamma(D)$.\label{eq:submer-Ric}
\end{thmlist}
\end{prop}

In particular, if $(M,g,\eta,\xi,\Phi)$ is \emph{Sasaki-Einstein}, then by~\ref{prop:Sasaki-Ric}~\ref{eq:submer-Ric-Reeb}
it has Einstein constant $n-1$, that is
\begin{equation}\label{eq:S-E}
\Ric_g =(n-1)g.
\end{equation}
Note that (\ref{eq:S-E}) is equivalent to $(C(M),\ol{g})$ being Ricci-flat, since
\[ \Ric_{\ol{g}} =\Ric_g -(n-1)g.\]

\subsubsection{3-Sasakian structures}

Recall that a hyperk\"{a}hler structure on a $4m$-dimensional manifold consists of a metric $g$ which is K\"{a}hler with respect to three complex structures $J_1, J_2, J_3$ satisfying the quaternionic relations $J_1 J_2 =-J_2 J_1 =J_3$ etc.
\begin{defn}\label{eq:3Sasaki}
A Riemannian manifold $(M,g)$ is 3-Sasakian if the metric cone $(C(M),\ol{g})$ is hyperk\"{a}hler, that is $\ol{g}$ admits compatible almost complex structures
$J_\alpha,\ \alpha=1,2,3 $ such that $(C(M),\ol{g},J_1,J_2,J_3)$ is a hyperk\"{a}hler structure.  Equivalently, $\Hol(C(M))\subseteq\Sp(m)$.
\end{defn}
A consequence of the definition is that $(M,g)$ is equipped with three Sasakian structures
$(g,\eta_i,\xi_i,\Phi_i),\ i=1,2,3$.
The Reeb vector fields $\xi_i =J_i(r\partial_r),\ i=1,2,3$ are orthogonal and satisfy $[\xi_i,\xi_j]=-2\varepsilon^{ijk}\xi_k$, where $\varepsilon^{ijk}$ is anti-symmetric in the indicies $i,j,k \in\{1,2,3\}$ and $\varepsilon^{123}=1$.
The tensors $\Phi_i ,\ i=1,2,3$ satisfy the identities
\begin{align}
\Phi_i(\xi_j) & =\varepsilon^{ijk}\xi_k  \label{eq:3-Sasak-id1} \\
\Phi_i \circ\Phi_j & =-\delta_{ij}\mathbb{1} +\varepsilon^{ijk}\Phi_k +\eta_j \otimes\xi_i \label{eq:3-Sasak-id2}
\end{align}
It is easy to see that there is an $\SR$ of Sasakian structures with Reeb vector field
$\xi_\tau =\tau_1 \xi_1 +\tau_2 \xi_2 +\tau_3 \xi_3$ with $\tau\in\SR$.

The Reeb vector fields $\{\xi_1,\xi_2,\xi_3 \}$ generate a Lie algebra $\mathfrak{sp}(1)$, so there is an effective isometric action
of either $\SO(3)$ or $\Sp(1)$ on $(M,g)$.  Both cases occur in the examples in this article.
This action generates a foliation $\mathscr{F}_{\xi_1,\xi_2,\xi_3}$ with generic leaves either $\SO(3)$ or $\Sp(1)$.

If we set $D_i =\ker\eta_i \subset TM,\ i=1,2,3$ to be the contact subbundles, then the complex structures
$J_i,\ i=1,2,3$ are recovered by
\begin{equation}
J_i (r\partial_r) =\xi_i,\quad J_i|_{D_i} =\Phi_i.
\end{equation}

Because a hyperk\"{a}her manifold is always Ricci-flat we have the following.
\begin{prop}
A 3-Sasakian manifold $(M,g)$ of dimension $4m-1$ is Einstein with Einstein constant $\lambda=4m-2$.
\end{prop}

We choose a Reeb vector field $\xi_1$, fixing a quasi-regular Sasakian structure, then the leaf space $\mathscr{F}_{\xi_1}$ is
a K\"{a}hler orbifold $Z$ with respect to the transversal complex structure $\ol{J}=\Phi_1$.
But it has in addition a complex contact structure and a fibering by rational curves which we now describe.
The 1-form $\eta^c =\eta_2 +\sqrt{-1}\eta_3$ is a $(1,0)$-form with respect to $\ol{J}$. But it is not invariant under
the $\U(1)$ group generated by $\exp(t\xi_1)$.  We have $\exp(t\xi_1)^* \eta^c =e^{2\sqrt{-1}\,t}\eta^c$.
Let $\mathbf{L}= M\times_{\U(1)}\C$, with $\U(1)$ acting on $\C$ by $e^{-2\sqrt{-1}\,t}$.
This is a holomorphic orbifold line bundle; in fact $C(M)$ is either $\mathbf{L}^{-1}$ or $\mathbf{L}^{-\frac{1}{2}}$
minus the zero section.  It is easy to see that each of these cases occur precisely where the Reeb vector fields generate
an effective action of $\SO(3)$ and $\Sp(1)$ respectively.  Then $\eta^c$ descends to an $\mathbf{L}$ valued holomorphic
1-form $\theta\in\Gamma\bigl(\Omega^{1,0}(\mathbf{L})\bigr)$.  It follows easily from (\ref{eq:3-Sasak-id2})
that $d\eta^c$ restricted to $D_1 \cap\ker\eta^c$ is a non-degenerate type $(2,0)$ form.  Thus $\theta$ is a complex contact
form on $Z$, and $\theta\wedge(d\theta)^{m-1} \in\Gamma\bigl(\mathbf{K}_{Z}\otimes\mathbf{L}^{m}\bigr)$ is
a non-vanishing section.  Thus $\mathbf{L}\cong\mathbf{K}_{Z}^{-\frac{1}{m}}$.

Each leaf of $\mathscr{F}_{\xi_1,\xi_2,\xi_3}$ descends to a rational curve in $Z$.  Each curve is a $\cps^1$
but may have orbifold singularities for non-generic leaves.  It is also well-known that restricted to a leaf
$\mathbf{L}|_{\cps^1} =\mathcal{O}(2)$, the degree 2 line bundle on a generic smooth leaf, while $\mathcal{O}(2)$ is interpreted as
an orbifold line bundle when the leaf has orbifold singularities.
The element $\exp(\frac{\pi}{2}\xi_2)$ acts on $M$ taking $\xi_1$ to $-\xi_1$, thus it descends to an anti-holomorphic
involution $\varsigma: Z\rightarrow Z$.  This \emph{real structure} is crucial to the twistor approach.
Note that $\varsigma^*\theta =\ol{\theta}$.
This all depends on the choice $\xi_1 \in\SR$ of the Reeb vector field.  But taking a different Reeb vector field gives an
isomorphic twistor space under the transitive action of $\Sp(1)$.

\section{Killing spinor deformations on Sasaki-Einstein manifolds}\label{sec:S-E-KS-def}

\subsection{Deformations of transversal complex structures}\label{subsec:def-versdef}

Let $(M,g,\eta,\xi,\Phi)$ be a Sasakian manifold.  Then the Reeb foliation $(\mathscr{F}_\xi,\ol{J})$ has
a transversely holomorphic structure.  The existence of a versal deformation space for $(\mathscr{F}_\xi,\ol{J})$, fixing
the smooth structure of $\mathscr{F}$, was proved in~\cite{ElKacNic89} and~\cite{Gir92} using arguments similar to those
in~\cite{Kur65}.

Let $\mathcal{A}^k =\Gamma(\Lambda^{0,k}_b\otimes\nu(\mathscr{F})^{1,0})$ be the space of smooth basic forms of
type $(0,k)$ with values in $\nu(\mathscr{F})^{1,0}$.  We have the Dolbeault complex
\begin{equation}\label{eq:Dol-comp}
 0\rightarrow \mathcal{A}^{0} \overset{\ol{\partial}_b}{\longrightarrow}\mathcal{A}^{1} \overset{\ol{\partial}_b}{\longrightarrow}\mathcal{A}^{2}\rightarrow\cdots.
\end{equation}
Here (\ref{eq:Dol-comp}) is the basic version of the complex used by Kuranishi~\cite{Kur65} whose
degree one cohomology is the space of first order deformation of the complex structure modulo diffeomorphisms.
Likewise, the first order deformations of $(\mathscr{F}_\xi,\ol{J})$ modulo foliate diffeomorphisms are given by
$H^1(\mathcal{A}^\bullet)$.  As in~\cite{Kur65} there is an open set $\mathcal{U}\subset H^1(\mathcal{A}^\bullet)$ and
the versal deformation space
$\mathcal{V}\subset\mathcal{U}$ is the germ of $\theta^{-1}(0)$ where $\theta$ is an analytic map
\[  H^1(\mathcal{A}^{\bullet}) \overset{\theta}{\rightarrow} H^2(\mathcal{A}^{\bullet}).\]

\begin{prop}\label{prop:vers-def}
Suppose $(M,g,\eta,\xi,\Phi)$ is Sasaki-Einstein (just $\Ric^T >0$ is sufficient).
We have $H^2(\mathcal{A}^\bullet)=\{0\}$, so the versal deformation space is smooth, $\mathcal{U}\subset H^1(\mathcal{A}^{\bullet})$.
\end{prop}
\begin{proof}
The basic version of Serre duality gives
\[  H^2(\mathcal{A}^{\bullet})= H^{m-3}_{\ol{\partial}_b}(\Gamma(\Lambda^{1,\bullet}_b \otimes\Lambda^{m-1,0}_b))=0,\]
where the second equality is given by
by Kodaira-Nakano vanishing, since $\Lambda^{m-1,0}_b <0$ and $(m-3)+1 =m-2 <m-1$.
The proof of Kodaira-Nakano vanishing in~\cite{GriHar78} goes through in transversally K\"{a}hler case using the
transversal harmonic theory of~\cite{ElKac90}.
\end{proof}

Since $\Ric^T >0$, the obstruction to lifting a deformation $\ol{J}_t,\ t\in\mathcal{U},$ to a deformation of
Sasakian structures vanish.
\begin{prop}\label{prop:vers-def-Sasaki}
Let $(M,g,\eta,\xi,\Phi)$ be Sasaki-Einstein (or just $\Ric^T >0$ is sufficient), then after possibly shrinking $\mathcal{U}$,
the deformation $\ol{J}_t,\ t\in\mathcal{U},$ lifts to a smooth family $(g_t, \eta_t,\xi,\Phi_t),\ t\in\mathcal{U},$
where $\Phi_t$ induces the transversal complex structure $\ol{J}_t$.
\end{prop}
\begin{proof}
We first show that the basic Dolbeault cohomology $H_b^{0,k} =H^k_{\ol{\partial}_b}(\Lambda^{0,\bullet}_b)=\{0\}$.
This can be proved using Kodaira vanishing as above or from the Weitzenb\"{o}ck formula on $\psi\in\Omega_b^{0,k}$
\begin{equation}\label{eq:Weit-Ric}
2\Delta_{\ol{\partial}_b}\psi_{\ol{\alpha}_1 \ldots\ol{\alpha}_k} =\ol{\Delta}^T \psi_{\ol{\alpha}_1 \ldots\ol{\alpha}_k}
+\sum_{j=1}^k(g^T)^{\beta\ol{\gamma}}\Ric^T_{\ol{\alpha}_j \beta}\psi_{\ol{\alpha}_1 \ldots\ol{\alpha}_{j-1}\ol{\gamma}\ol{\alpha}_{j+1}\ldots\ol{\alpha}_k},
\end{equation}
where $\ol{\Delta}^T =(\nabla^T)^* \nabla^T$ is the transversal rough Laplacian.
Then if $\psi$ is harmonic and $\Ric^T \geq\lambda g^T$ then integrating (\ref{eq:Weit-Ric}) gives
\[ 0\geq\int_M \bigl(\langle\nabla^T \psi,\nabla^T \psi\rangle +k\lambda\langle\psi,\psi\rangle\bigr)\,\mu_g,\]
where $\langle\cdot,\cdot\rangle$ is the Hermitian product and $\mu_g =\frac{1}{(m-1)!}\eta\wedge(\frac{1}{2}d\eta)^{m-1}$.
Therefore $\psi=0$.

By~\cite{ElKacGmi97} there is a family of transversal K\"{a}hler metrics with K\"{a}hler forms $\omega^T_t$ on
$(\mathscr{F}_\xi,\ol{J}_t)$ depending smoothly on $t\in\mathcal{U}$ with $\omega^T_0=\omega^T$.
The above argument shows that after shrinking
$\mathcal{U}$ the Dolbeault groups on $(\mathscr{F}_\xi,\ol{J}_t)$ also satisfy $H_{b,t}^{0,k}=\{0\}$.
Since the harmonic space $\mathcal{H}^2_{\Delta_{\ol{\partial}_{b,t}}}$, of the transverse Laplacian $\Delta_{\ol{\partial}_{b,t}}$
with respect to $\omega^T_t$, has constant dimension, by for example~\cite[Lemma 4.3]{Koi83}
there are isomorphisms $R_t :\mathcal{H}^2_{\Delta_{\ol{\partial}_{b}}}\rightarrow\mathcal{H}^2_{\Delta_{\ol{\partial}_{b,t}}}$
depending smoothly on $t$.  There exists smoothly varying $\alpha_t\in\mathcal{H}^2_{\Delta_{\ol{\partial}_{b}}}$ so that
$R_t(\alpha_t) =[\omega^T -\omega^T_t]_h$, the harmonic component.  Let $G$ be the Green's operator for
$\Delta_{\ol{\partial}_{b}}$.  Let $\beta_t =d^* G(\omega_t^T + R_t(\alpha_t) -\omega^T)$, and define
$\eta_t =\eta +\beta_t$.  Then $\frac{1}{2}d\eta_t =\omega_t^T + R_t(\alpha_t)$ which is of type $(1,1)$ and is
positive definite for small enough $t$.

The family of Sasakian structures $(g_t,\eta_t,\xi,\Phi_t)$ is defined by lifting $\ol{J}_t$ to $\ker\eta_t$ to get
$\Phi_t$, while
\begin{equation}\label{eq:Sasak-def}
g_t =\frac{1}{2}d\eta_t (\cdot,\Phi_t \cdot) +\eta_t \otimes\eta_t.
\end{equation}
\end{proof}

\begin{remark}
With the assumption $c_1(\mathscr{F}_\xi,\ol{J}_t) >0$ made in this article, the deformations in Proposition~\ref{prop:vers-def-Sasaki}
along with transversal K\"{a}hler deformations
\[ \tilde{\eta} =\eta+ d^c \varphi,\quad \tilde{\Phi} =\Phi -\xi\otimes\tilde{\eta}\circ\Phi, \]
for $\varphi\in C^\infty_b(M)$ basic, give all local deformations of the Sasakian structure fixing the Reeb vector field.
See~\cite{vanCoTip15} for details.
\end{remark}

Since a Sasaki-Einstein structure is transversally K\"{a}hler-Einstein by Proposition~\ref{prop:Sasaki-Ric}.\ref{eq:submer-Ric},
a necessary condition for a compatible Sasaki-Einstein structure is that
\begin{equation}\label{eq:KE-cond}
\pi c_1(\mathscr{F}_\xi,\ol{J})=m\omega^T.
\end{equation}
It follows from the proof of Proposition~\ref{prop:vers-def-Sasaki} that if (\ref{eq:KE-cond}) holds for
$(M,g,\eta,\xi,\Phi)$, then the family $(g_t, \eta_t,\xi,\Phi_t),\ t\in\mathcal{U},$ also satisfies
\[ \pi c_1(\mathscr{F}_\xi,\ol{J}_t)=m\omega_t^T. \]

We consider some properties of a first order deformation through Sasakian metrics which will be used later.
We differentiate (\ref{eq:Sasak-def}) and use the notation
\[ \dot{\ol{J}}_t =I,\quad \dot{\omega}^T =\phi,\ \text{ and }\ \dot{g}^T =h\]
where we have
\begin{equation}
d\dot{\eta}=2\phi.
\end{equation}
Since $\omega_t^T(X,Y) =g^T_t(\ol{J}_t X, Y)$, we have
\begin{gather}
\phi_{\alpha\beta} =\sqrt{-1}h_{\alpha\beta} +I_{\alpha\beta}\label{eq:Sasak-def1} \\
\phi_{\alpha\ol{\beta}} = \sqrt{-1}h_{\alpha\ol{\beta}}.\label{eq:Sasak-def2}
\end{gather}
Note that since $I$ anti-commutes with $\ol{J}_0$, it only has components $\tensor{I}{_\alpha^{\ol{\beta}}}$ and
$\tensor{I}{_{\ol{\alpha}}^\beta}$.

In addition differentiating
\begin{equation}
g^T_t(\ol{J}_t X,Y)+ g^T_t(X,\ol{J}_t Y)=0
\end{equation}
gives
\begin{equation}\label{eq:Sasak-def3}
2\sqrt{-1}h_{\alpha\beta} + (I_{\alpha\beta} +I_{\beta\alpha}) =0.
\end{equation}
Finally (\ref{eq:Sasak-def1}) and (\ref{eq:Sasak-def3}) give
\begin{equation}\label{eq:Sasak-def4}
\phi_{\alpha\beta} =\frac{1}{2}(I_{\alpha\beta} -I_{\beta\alpha}).
\end{equation}

\subsection{Skew-Hermitian Einstein deformations}

By Proposition~\ref{prop:Sasaki-Ric}.\ref{eq:submer-Ric} if $(M,g,\eta,\xi,\Phi)$ is Sasaki-Einstein then
the transversal K\"{a}hler metric $g^T$ on $\mathscr{F}_\xi$ is Einstein
\[ \Ric_{g^T} =2m g^T.\]

We define the space $\EED(g^T)$ just as in Section~\ref{subsec:Einst-def} using the transversal Levi-Civita connection
defined in (\ref{eq:LC-trans}), that is
\[\EED(g^T) =\{h\in\Gamma\bigl(\Sym^2 T_b^*M \bigr)\, |\,\tr_{g^T} h  =\delta_{g^T} h=0,\ \bigl(\ol{\Delta}^T +2L^T \bigr) h=0\}, \]
where $L^T$ is defined as in (\ref{eq:Einst-lin}) but with the transverse curvature $R^T$.

Given $h\in\Gamma\bigl(\Sym^2 T_b^*M \bigr)$ we decompose $h$ into its Hermitian $h_H$ and anti-Hermitian $h_A$ parts
with respect to the transversal complex structure $\ol{J}$ on $\nu(\mathscr{F}_\xi)$, i.e.
\[ h_H(\ol{J}X,\ol{J}Y) = h_H(X,Y),\quad h_A(\ol{J}X,\ol{J}Y) =-h_A(X,Y). \]

We denote by $\EED_H(g^T)$ (resp. $\EED_A(g^T)$) the space of Hermitian (resp. anti-Hermitian) essential infinitesimal Einstein
deformations.  The following is an adaptation of results of N. Koiso~\cite{Koi83} to the current situation.
\begin{prop}\label{prop:Herm-def}
Suppose $(M,g,\eta,\xi,\Phi)$ is Sasaki-Einstein.  Then we have the decomposition
\begin{equation}\label{eq:Her-ant}
\EED(g^T) =\EED_H(g^T)\oplus \EED_A(g^T),
\end{equation}
and $h\in\Gamma\bigl(\Sym^2 \Lambda_b^{0,1} \bigr)$ is an element of $\EED_A(g^T)$ if and only if
\begin{align}\label{eq:harm-ident1}
\nabla^T_{\ol{\alpha}} h_{\ol{\beta}\ol{\gamma}} -\nabla^T_{\ol{\beta}}h_{\ol{\alpha}\ol{\gamma}} & = 0 \\
(\nabla^T)^{\ol{\alpha}} h_{\ol{\alpha}\ol{\beta}} & =0 \label{eq:harm-ident2}.
\end{align}
\end{prop}
\begin{proof}
Suppose $h\in\Gamma\bigl(\Sym^2 \Lambda_b^{0,1} \bigr)$.  If $h^\sharp$ denotes raising the second index, then
$h^\sharp\in\mathcal{A}^1$.  We have the Weitzenb\"{o}ck formula
\begin{equation}\label{eq:Weit-Lap}
\ol{\partial}_b \ol{\partial}_b^* h^\sharp +\ol{\partial}_b^* \ol{\partial}_b h^\sharp =\frac{1}{2}\bigl(\ol{\Delta}^T +2L^T \bigr)h^\sharp.
\end{equation}

Suppose $h\in\EED(g^T)$.  Then $\bigl(\ol{\Delta}^T +2L^T \bigr)h_A =0$ and (\ref{eq:Weit-Lap}) implies
$\delta_{g^T}h_A =0$.  Trivially, $\tr_{g^T} h_A =0$.  Thus $h_A \in\EED(g^T)$ and (\ref{eq:Her-ant}) follows.

It follows from (\ref{eq:Weit-Lap}) that $h\in\Gamma\bigl(\Sym^2 \Lambda_b^{0,1} \bigr)$ is in $\EED_A(g^T)$
if and only if (\ref{eq:harm-ident1}) and (\ref{eq:harm-ident2}) hold.
\end{proof}

Let $\mathcal{H}_{\mathcal{A}}^k$ denote the k-th harmonic space of the complex (\ref{eq:Dol-comp}).
\begin{cor}\label{cor:comp-Einst-def}
Let $(M,g,\eta,\xi,\Phi)$ be Sasaki-Einstein.  Then there is a canonical isomorphism
\begin{equation}
\begin{array}{rcl}
\mathcal{H}_{\mathcal{A}}^1 & \simarrow & \EED_A(g^T) \\
h_{\ol{\alpha}\ol{\beta}} & \longmapsto & -\sqrt{-1}h_{\ol{\alpha}\ol{\beta}}.
\end{array}
\end{equation}
\end{cor}
\begin{proof}
First note that from Proposition~\ref{prop:Herm-def} and formula (\ref{eq:Weit-Lap}) we have a decomposition
\begin{equation}
\mathcal{H}_{\mathcal{A}}^1 =\mathcal{H}_{\mathcal{A},S}^1 \oplus\mathcal{H}_{\mathcal{A},A}^1,
\end{equation}
into symmetric and anti-symmetric parts.
If $\phi\in\mathcal{H}_{\mathcal{A},A}^1$ then $L\phi =0$.  Thus (\ref{eq:Weit-Lap}) shows that
$\ol{\Delta}^T \phi =0$, and we have $\nabla^T \phi =0$.  Lowering an index gives an harmonic
$\phi_{\ol{\alpha}\ol{\beta}} \in\Omega_b^{0,2}$.  Since $M$ is Sasaki-Einstein (\ref{eq:Weit-Ric}) becomes
\[ 2\Delta_{\ol{\partial}_b}\phi_{\ol{\alpha}\ol{\beta}}=\ol{\Delta}^T \phi_{\ol{\alpha}\ol{\beta}}+4m \phi_{\ol{\alpha}\ol{\beta}}.\]
Since all but the last term are zero, $\phi_{\ol{\alpha}\ol{\beta}}=0$.
\end{proof}

\begin{lem}\label{lem:Einst-def-sub}
Let $(M,g,\eta,\xi,\Phi)$ be Sasaki-Einstein and $h^T\in\Gamma\bigl(\Sym^2 T_b^*M \bigr)$ an element of $\EED_A(g^T)$.
If $h=\pi^* h^T$ is the pull-back of the basic tensor $h^T$ to $M$ then $h\in\EED(g)$.
\end{lem}
\begin{proof}
First note that the O'Neill tensor of the local projection $\pi$ onto the leaf space of the foliation $\mathscr{F}_{\xi}$ is
\begin{equation}
A_X Y =g(\xi ,\nabla_X Y)\xi =-g(\Phi X,Y)\xi, \quad X,Y \in\Gamma(D).
\end{equation}
We will use the formulae of O'Neill on the curvature of a Riemannian submersion.  See~\cite[ch. 9]{Bes87} for more details.

If $X,Y,Z,W \in\Gamma(D)$ are basic vector fields, then we have
\begin{equation}
\begin{split}\label{eq:subm-curv1}
g(R(X,Y)Z,W) & = g^T(R^T(X,Y)Z,W) +2g(\Phi X,Y)g(\Phi Z,W)+g(\Phi X,Z)g(\Phi Y,W)  \\
				& -g(\Phi Y,Z)g(\Phi X,W),\\	
\end{split}
\end{equation}
\begin{equation}\label{eq:subm-curv2}
g(R(X,Y)\xi, W) = g(X,W)g(Y,\xi) -g(X,\xi)g(Y,W).
\end{equation}
A routine calculation shows that
\begin{align*}
\ol{\Delta} h(X,Y) & =\pi^*\bigl(\ol{\Delta}^T h^T \bigr)(X,Y) +4h(X,Y)-2h(\Phi X,\Phi Y), \\
\ol{\Delta} h(\xi,X) & =-2\delta h^T (\Phi X),\\
\ol{\Delta} h(\xi,\xi) & =-2\tr h^T.
\end{align*}

We compute from (\ref{eq:subm-curv1}) using an orthonormal frame $\{e_1 ,\ldots, e_{2m-2}, \xi\}$ that
\begin{equation}
\begin{split}
Lh(X,Y) = & \pi^*\bigl(L^T h^T\bigr)(X,Y) +\sum_{i,j}\Bigl[ 2g(\Phi X,e_i)g(\Phi Y,e_j)+g(\Phi X,Y)g(\Phi e_i, e_j) \\
			&-g(\Phi e_i,Y)g(\Phi X,e_j) \Bigr]h(e_i,e_j) \\
			& =\pi^*\bigl(L^T h^T\bigr)(X,Y) +2h(\Phi X,\Phi Y) +h(\Phi Y,\Phi X) \\
			& =\pi^*\bigl(L^T h^T\bigr)(X,Y) -3h(X,Y).\\
\end{split}			
\end{equation}
And (\ref{eq:subm-curv2}) easily gives
\begin{equation}
Lh(X,\xi)=-g(\xi,X)\tr h + h(\xi,X)=0.
\end{equation}

It follows from the above equations that
\begin{equation}
\Bigl(\ol{\Delta}+2L\Bigr)h =\pi^*\bigl(\ol{\Delta}^T h^T \bigr)+2\pi^*\bigl(L^T h^T\bigr),
\end{equation}
and $\delta h =0$, $\tr h =0$ are trivial.
\end{proof}
\begin{remark}
It is clear from the proof that a non-zero $h=\pi^*h^T$ is not an infinitesimal Einstein deformation if $h^T$ is not anti-Hermitian.
\end{remark}

\subsection{Infinitesimal deformations on Sasaki-Einstein manifolds}

From Proposition~\ref{prop:Herm-def} and Lemma~\ref{lem:Einst-def-sub} for any $\beta\in\mathcal{H}^1_{\mathcal{A}}$
we have $h^\beta \in\EED(g)$, where $h^\beta (X,Y)=g^T(\ol{J}\beta X, Y)$.   We define as in Section~\ref{subsubsec:spin-def}
$\Psi^{\beta,\sigma_0}(X)=\alpha(X)\sigma_0$, where $\alpha =-\frac{1}{2}(h^\beta)^\sharp$ and $\sigma_0$ is Killing spinor.

\begin{prop}\label{prop:inf-def-Sasaki}
Let $(M,g)$ be a spin Sasaki-Einstein manifold admitting the 2 defining Killing spinors $\sigma_j,\ j=0,1$.
If $\beta\in\mathcal{H}^1_{\mathcal{A}}$ then the corresponding basic anti-Hermitian symmetric tensor
$h^\beta$ is an infinitesimal Einstein deformation of $g$, and $(\alpha,0),\ \alpha =-\frac{1}{2}(h^\beta)^\sharp$ is
an infinitesimal deformation of the Killing spinors $\sigma_j$ for $j=0,1$.
\end{prop}
\begin{remark}
The definitions of $h^\beta,\,\Psi^{\beta,\sigma_0}$ and $\alpha$ are made to agree with the identifications made in
Corollary~\ref{cor:comp-Einst-def} and Section~\ref{subsubsec:spin-def}.
\end{remark}

\begin{proof}
That $h^\beta$ is an infinitesimal Einstein deformation follows from Lemma~\ref{lem:Einst-def-sub}.

In the proof we denote $(h^\beta)^\sharp$ by $h$ which can be considered to be a basic tensor with values in
$D=\ker\eta$ and $\Phi h=-h\Phi$.
By Proposition~\ref{prop:Kill-spin-inf} it is sufficient to prove
\begin{equation}\label{eq:Sasak-Kill-def1}
\sum_i e_i \cdot\bigl(\nabla_i h\bigr)(X)\sigma_j =2ch(X)\sigma_j ,\quad\text{for all } X\in TM,\ j=0,1.
\end{equation}
for a local orthonormal frame $\{e_1,\ldots, e_{2m-1}\}$ for which we may choose $e_i \in\Gamma\bigl(D\bigr)$ for
$i=1,\ldots,2m-2,\ e_{m-1+i} =\Phi e_i$ for $i=1,\ldots,m-1$ and $e_{2m-1} =\xi$.
We extend to an orthonormal frame on $C(M)$ by setting $e_{2m} =\partial_r$.

Define an Hermitian frame by $\varepsilon_\alpha =\frac{1}{\sqrt{2}}(e_\alpha -\sqrt{-1}Je_\alpha),\ \alpha=1,\ldots,m-1$ and
$\varepsilon_{m} =\frac{1}{\sqrt{2}}(e_{2m-1} -\sqrt{-1}Je_{2m-1})=\frac{1}{\sqrt{2}}(\xi +\sqrt{-1}\partial_r)$.
Denote their duals by $\varepsilon^\alpha =\frac{1}{\sqrt{2}}(e_\alpha +\sqrt{-1}Je_\alpha)$ and define $\varepsilon_{\ol{\alpha}} =\ol{\varepsilon}_\alpha$.  Note that $\varepsilon_{\ol{\alpha}} = \varepsilon^\alpha$.

Since $Hol(\ol{g})\subseteq\SU(m)$ the spinor bundle $\Sigma$ of $M$ can be identified, on the neighborhood of the frame, with
$\Lambda^{ev} \Span_{\C}\{\varepsilon_\alpha |\alpha=1,\ldots,m\} =\Lambda^{ev} T^{1,0}C(M)|_M$,
or $\Lambda^{odd} \Span_{\C}\{\varepsilon_\alpha |\alpha=1,\ldots,m\}$.
Clifford multiplication is given by $e_i \mapsto e_i e_{2m},\ 1\leq i\leq 2m-1$ (or $e_i \mapsto -e_i e_{2m}$ giving
the other Clifford module structure on $\Sigma$).

If m is even we take $\Sigma =\Lambda^{ev} \Span_{\C}\{\varepsilon_\alpha |\alpha=1,\ldots,m\}$.  If m is odd, then
we take $\Sigma =\Lambda^{odd} \Span_{\C}\{\varepsilon_\alpha |\alpha=1,\ldots,m\}$ when considering
$\sigma_1 \in\Gamma(\Sigma)$, and $\Sigma =\Lambda^{even} \Span_{\C}\{\varepsilon_\alpha |\alpha=1,\ldots,m\}$
when considering $\sigma_0 \in\Gamma(\Sigma)$.  In the latter case we take Clifford multiplication to act through
$e_i \mapsto -e_i e_{2m}$ in order to obtain the same Clifford module structure on $\Sigma$ (in this case $c=-\frac{1}{2}$).

The Killing spinors are locally $\sigma_0 =a(x)\in\Gamma\bigl(\Lambda^0 \bigr)$ and $\sigma_1 =b(x)\varepsilon_1 \wedge\cdots\wedge\varepsilon_{m} \in\Gamma\bigl(\Lambda^{m}\bigr)$, where $a,b$ are smooth functions.

Note that for $X,Y\in\Gamma\bigl(D\bigr)$ basic
\begin{equation}
\begin{split}
\nabla_Y h(X) & =\nabla_Y^T h(X) +g(\nabla_Y (hX),\xi)\xi \\
				& =\nabla_Y^T h(X) -g(h(X),\Phi Y)\xi. \\
\end{split}
\end{equation}
Thus
\begin{equation}\label{eq:Sasak-Kill-def2}
\begin{split}
\sum_{i=1}^{2m-1} e_i\bigl(\nabla_i h\bigr)(X)\sigma_j & = \sum_{i=1}^{2m-2} e_i\bigl(\nabla^T_i h\bigr)(X)\sigma_j +\xi\bigl(\nabla_\xi h\bigr)(X)\sigma_j +\sum_i e_i g(\Phi h(X),e_i)\xi\sigma_j \\
											& = \sum_{i=1}^{2m-2} e_i \bigl(\nabla^T_i h\bigr)(X)\sigma_j +2\xi\Phi h(X)\sigma_j+\Phi h(X)\xi\sigma_j \\
											& = \sum_{i=1}^{2m-2} e_i \bigl(\nabla^T_i h\bigr)(X)\sigma_j -\Phi h(X)\xi\sigma_j. \\
\end{split}
\end{equation}
We will show that the first term on the right of (\ref{eq:Sasak-Kill-def2}) vanishes.  First suppose $X=\varepsilon_\gamma$, then
\begin{equation}\label{eq:Sasak-Kill-def3}
\begin{split}
 \sum_{i=1}^{2m-2} e_i \bigl(\nabla^T_i h\bigr)(X)\sigma_j & =\sum_{\alpha=1}^{m-1} \varepsilon^\alpha \bigl(\nabla^T_{\varepsilon_\alpha} h\bigr)(\varepsilon_\gamma)\sigma_j +
 \sum_{\alpha=1}^{m-1} \varepsilon^{\ol{\alpha}} \bigl(\nabla^T_{\varepsilon_{\ol{\alpha}}} h\bigr)(\varepsilon_\gamma)\sigma_j \\
 				& = \varepsilon^\alpha \nabla^T_\alpha {h_{\gamma}}^{\ol{\beta}}\varepsilon_{\ol{\beta}} \sigma_j +
 				 \varepsilon^{\ol{\alpha}} \nabla^T_{\ol{\alpha}} {h_{\gamma}}^{\ol{\beta}}\varepsilon_{\ol{\beta}} \sigma_j .\\
\end{split}
\end{equation}
If $j=0$, then this vanishes since $\varepsilon_{\ol{\beta}}\sigma_0=0$.  Suppose $j=1$, then the first term on the right of (\ref{eq:Sasak-Kill-def3}) is
\begin{equation}
\begin{split}
\varepsilon^\alpha \nabla^T_\alpha {h_{\gamma}}^{\ol{\beta}}\varepsilon_{\ol{\beta}} \sigma_1 & =\nabla^T_\alpha h_{\beta\gamma}\varepsilon^\alpha\varepsilon^\beta \sigma_1 \\
					& = \sum_{\alpha<\beta}\bigl(\nabla^T_\alpha h_{\beta\gamma} -\nabla^T_\beta h_{\alpha\gamma}\bigr)\varepsilon^\alpha\varepsilon^\beta \sigma_1=0, \\
\end{split}
\end{equation}
because of (\ref{eq:harm-ident1}).
And the second term on the right of (\ref{eq:Sasak-Kill-def3}) is
\begin{equation}
\begin{split}
\varepsilon^{\ol{\alpha}} \nabla^T_{\ol{\alpha}} {h_{\gamma}}^{\ol{\beta}}\varepsilon_{\ol{\beta}} \sigma_1 & =\nabla^T_{\ol{\alpha}} h_{\beta\gamma}\varepsilon^{\ol{\alpha}}\varepsilon^\beta \sigma_1 \\
					& = \nabla^T_{\ol{\alpha}} h_{\beta\gamma}\bigl(-\varepsilon^\beta\varepsilon^{\ol{\alpha}}-2g(\varepsilon^{\ol{\alpha}},\varepsilon^\beta)\bigr) \sigma_1 \\
					& = -2\bigl(\nabla^T\bigr)^\alpha h_{\alpha\gamma} \sigma_1 =0, \\				
\end{split}
\end{equation}
because of (\ref{eq:harm-ident2}).
The case of $X=\varepsilon_{\ol{\gamma}}$ is completely analogous.

We have
\begin{equation}\label{eq:Sasak-kill-def6}
\begin{split}
\sum_{i=1}^{2m-1} e_i\bigl(\nabla_i h\bigr)(\xi)\sigma_j & =-\sum_{i=1}^{2m-2} e_i h(\Phi e_i)\sigma_j \\
														 & =-\sum_{i,k=1}^{2m-2} e_i h(\Phi e_i,e_k)e_k \sigma_j \\
														 & =\sum_{i=1}^{2m-2} h(\Phi e_i,e_i)\sigma_j =0,
\end{split}
\end{equation}
for $j=0,1$.  The last two equalities follow because $h(\Phi\cdot,\cdot)$ is symmetric and anti-Hermitian.

We have that
\begin{equation}\label{eq:Sasak-kill-def4}
\sum_{i=1}^{2m-1} e_i\bigl(\nabla_i h\bigr)(X)\sigma_j = -\Phi h(X)\xi\sigma_j, \quad\text{for } X\in TM.
\end{equation}
Recall that Clifford multiplication is $X\cdot\sigma_j =X\partial_r \sigma_j$, for $X\in TM$ with our representation space, unless $\sigma_j $ has $c=-\frac{1}{2}$ in which case we must take $X\cdot\sigma_j =-X\partial_r \sigma_j$.
It is easy to check that
\begin{equation}\label{eq:Sasak-kill-def5}
-\Phi h(X)\xi\sigma_j =h(X)\partial_r \sigma_j, \quad j=1,2.
\end{equation}

Then (\ref{eq:Sasak-Kill-def1}) follows from (\ref{eq:Sasak-kill-def4}) and (\ref{eq:Sasak-kill-def5}) and the Proposition
follows.
\end{proof}

\subsection{Infinitesimal deformations on 3-Sasakian manifolds}

Recall the important result of H. Pedersen and Y. S. Poon that 3-Sasakian structures are rigid.
\begin{thm}[\cite{PedPo99}]\label{thm:PedPo}
Let $(M,g),\ \dim M=4m-1,$ be a 3-Sasakian manifold with Killing spinors $\sigma_i,\ i=0,\ldots, m$.  Then any Einstein deformation $(M,g_t)$ of $g$ with compatible 3-Sasakian structures, i.e. preserving the existence of the $\sigma_i,\ i=0,\ldots, m$, is trivial.  That is, there exists a family $\phi_t$ of diffeomorphisms of $M$ with $\phi_t^*g_t =g$.
\end{thm}

The transversal space $\mathscr{F}_\xi$, for any fixed Reeb vector field $\xi\in\SR$, is an orbifold $Z$ with a complex contact
structure.  Recall that the twistor spaces for any
two $\xi\in\SR$ are isomorphic via the transitive action of $\Sp(1)$ on the $\SR$ of Reeb vector fields.
We denote by $\mathcal{H}^1_{\mathcal{A}}(\xi)$ the harmonic space of the particular $\xi\in\SR$.  Although, the
$\mathcal{H}^1_{\mathcal{A}}(\xi),\ \xi\in\SR,$ are isomorphic they give different deformations
$h^\beta\in\EED(g),\ \beta\in \mathcal{H}^1_{\mathcal{A}}(\xi)$.

The proof of Theorem~\ref{thm:PedPo}, and the earlier similar result~\cite{LeB88} of C. LeBrun, follow mainly from the vanishing of $H^1(Z,\mathcal{O}(\mathbf{L}))$. We have
\[ H^1(Z,\mathcal{O}(\mathbf{L})) =H^1(Z,\Omega^{2m-1}(\mathbf{K}_Z^{-1}\otimes\mathbf{L}))=\{0\}\]
by Kodaira vanishing, since $\mathbf{K}_Z^{-1}\otimes\mathbf{L} >0$.

The following provides a spinor version of this vanishing result.
\begin{prop}\label{prop:inf-def-3Sasaki}
Let $(M,g),\ \dim M=4m-1,$ be a 3-Sasakian manifold with Killing spinors $\sigma_j,\ j=0,\ldots, m$.
If $\beta\in\mathcal{H}^1_{\mathcal{A}}(\xi)$ is nonzero, then the corresponding basic anti-Hermitian symmetric tensor
$h^\beta$ is an infinitesimal Einstein deformation of $g$, and $(\alpha,0),\ \alpha =-\frac{1}{2}(h^\beta)^\sharp$ is
an infinitesimal deformation of the Killing spinors $\sigma_j$ for $j=0,m$, but never for any nonzero
$\sigma\in\Span_{\C}\{\sigma_j | j=1,\ldots,m-1 \}$.
\end{prop}

It will be convenient to introduce some notation.  Given $\sigma\in\mathcal{N}_g$ we change notation and write the formula in
Proposition~\ref{prop:Kill-spin-inf} as
\begin{equation}\label{eq:Kill-spin-inf}
 \mathcal{L}(\alpha,\sigma)(X) =-\frac{1}{2} \sum_i e_i \cdot\bigl(\nabla_i \alpha\bigr)(X)\sigma + \frac{1}{2}\alpha(X)\sigma ,\quad\text{for all } X\in TM.
\end{equation}
Then the proposition asserts that $\mathcal{L}(\alpha,\sigma)=0$ for $\sigma=\sigma_j \ j=0,m$ and
$\mathcal{L}(\alpha,\sigma)\neq 0$ for nonzero $\sigma\in\Span_{\C}\{\sigma_j | j=1,\ldots,m-1 \}$.

\begin{proof}
We consider a local orthonormal frame which is in the $\Sp(m)$-structure of $C(M)$
\[ (e_1, e_2, \ldots ,e_{4m}) =(f_1, J_1 f_1,J_2 f_1,J_3 f_1; f_2,J_1 f_2 \ldots; f_m, J_1 f_m, J_2 f_m, J_3 f_m), \]
where $e_1,\ldots, e_{4m-4} \in\cap_{i=1,2,3} D_i =\mathcal{D},\ f_m =-\xi_3, J_1 f_m =\xi_2, J_2 f_m =-\xi_1$ and $J_3 f_m =\partial_r$.

We define an Hermitian frame by $\varepsilon_\alpha =\frac{1}{\sqrt{2}}(e_{2\alpha-1}-\sqrt{-1}J_1 e_{2\alpha-1})=
\frac{1}{\sqrt{2}}(e_{2\alpha-1}-\sqrt{-1} e_{2\alpha}),$ $\alpha=1,\ldots, 2m$, and
their duals $\varepsilon^{\alpha}=\varepsilon_{\ol{\alpha}}=\ol{\varepsilon}_\alpha$.
In particular, we have
$\varepsilon_{2m-1} =\frac{1}{\sqrt{2}}(-\xi_3 -\sqrt{-1}\xi_2)$ and $\varepsilon_{2m}=\frac{1}{\sqrt{2}}(-\xi_1 -\sqrt{-1}\partial_r)$.
As in the proof of Proposition~\ref{prop:inf-def-Sasaki} the spinor bundle of $(M,g)$ can be identified with
$\Sigma =\Lambda^{ev} T^{1,0} C(M)|_M = \Lambda^{ev} \Span_{\C}\{\varepsilon_\alpha |\alpha=1,\ldots,2m\}$.

Define the ``symplectic form''
\begin{equation}\label{eq:sym-form}
\varpi =\sum_{\alpha=1}^m \varepsilon_{2\alpha -1}\wedge\varepsilon_{2\alpha}.
\end{equation}
The Killing spinors on $(M,g)$ can be identified with
\[ \sigma_k =\frac{1}{k!}\varpi^k\in\Gamma\bigl( \Lambda^{ev} T^{1,0} C(M)\bigr),\ k=0,\ldots,m. \]
From the proof of Proposition~\ref{prop:inf-def-Sasaki} a Killing spinor $\sigma_k$ is preserved to first order by the Einstein deformation $h$ if and only if
\begin{equation}\label{eq:3Sasak-def1}
\sum_{\alpha=1}^{2m-1} \varepsilon^\alpha \nabla_\alpha^T h(X) \sigma_k +\sum_{\alpha=1}^{2m-1}\varepsilon^{\ol{\alpha}}\nabla_{\ol{\alpha}}^T h(X)\sigma_k +
\xi_1 \Phi_1 h(X)\sigma_k =2ch(X)\sigma_k,
\end{equation}
holds for all $X\in D_1$.  Here $c=\frac{1}{2}$.

Define $\psi\in\Omega^{0,1}(\mathbf{L})$ by $\psi_{\ol{\beta}} =h_{\ol{\beta}}^\gamma \theta_\gamma$.  Since $\theta$ is holomorphic
$\ol{\partial}\psi =0$.  The line bundle $\mathbf{L}$ has a natural hermitian metric by the identification $\mathbf{L}=\mathbf{K}_Z^{-\frac{1}{m}}$,
so there is a natural connection on $\mathbf{L}$.  Then
\begin{equation}
\begin{split}
\ol{\partial}^* \psi & =-\nabla^{T\ol{\beta}}\psi_{\ol{\beta}} \\
					 & =-{h_{\ol{\beta}}}^\gamma \nabla^{\ol{\beta}}\theta_{\gamma} =0,\\
\end{split}
\end{equation}
where the second equality holds from $\nabla^{T\ol{\beta}}{h_{\ol{\beta}}}^\gamma =0$.  For the third equality observe that
$\nabla_{\beta} \theta_\gamma$ lifts to the form $d\eta^c= g(\Phi_2 \cdot,\cdot)+\sqrt{-1}g(\Phi_3 \cdot,\cdot)$ restricted to $D_1$,
but $h^{\beta\gamma}$ is symmetric and so the contraction is zero.

Therefore $\psi\in\Omega^{0,1}(\mathbf{L})$ is harmonic.  But as we observed, $H^1(Z,\mathcal{O}(\mathbf{L}))=0$, so $\psi=0$.
It follows that $h(X)\in\mathcal{D}$ for all $X\in TM$.  This fact will be used repeatedly in the rest of the proof.

Substituting $\xi_1 =\frac{-1}{\sqrt{2}}(\varepsilon_{2m} +\varepsilon_{\ol{2m}})$ and $\partial_r=\frac{\sqrt{-1}}{\sqrt{2}}(\varepsilon_{2m} -\varepsilon_{\ol{2m}})$
into (\ref{eq:3Sasak-def1}) and canceling terms gives
\begin{equation}\label{eq:3Sasak-def2}
\sum_{\alpha=1}^{2m-1} \varepsilon^\alpha \nabla_\alpha^T h(X) \sigma_k +\sum_{\alpha=1}^{2m-1}\varepsilon^{\ol{\alpha}}\nabla_{\ol{\alpha}}^T h(X)\sigma_k -
\sqrt{-1}\sqrt{2}h(X)^{0,1} \varepsilon_{2m} +\sqrt{-1}\sqrt{2}h(X)^{1,0}\varepsilon_{\ol{2m}}=0.
\end{equation}

We saw in the proof of Proposition~\ref{prop:inf-def-Sasaki} that
\[\varepsilon^\alpha \nabla^T_\alpha h_{\beta\gamma}\varepsilon^\gamma =\varepsilon^{\ol{\alpha}} \nabla^T_{\ol{\alpha}} h_{\ol{\beta}\ol{\gamma}}\varepsilon^{\ol{\gamma}} =0,\]
so (\ref{eq:3Sasak-def2}) becomes
\begin{align}
\sum_{\alpha=1}^{2m-1}\sum_{\gamma=1}^{2m-1} \nabla_{\ol{\alpha}}^T h_{\beta\gamma}\varepsilon^{\ol{\alpha}}\varepsilon^{\gamma}\sigma_k & -\sqrt{-1}\sqrt{2}h(\varepsilon_\beta)\varepsilon_{2m} \sigma_k =0, \text{ for } X=\varepsilon_\beta, \label{eq:3Sasak-def3a} \\
\sum_{\alpha=1}^{2m-1}\sum_{\gamma=1}^{2m-1} \nabla_{\alpha}^T h_{\ol{\beta}\ol{\gamma}}\varepsilon^{\alpha}\varepsilon^{\ol{\gamma}}\sigma_k & +\sqrt{-1}\sqrt{2}h(\varepsilon_{\ol{\beta}})\varepsilon_{\ol{2m}} \sigma_k =0, \text{ for } X=\varepsilon_{\ol{\beta}}. \label{eq:3Sasak-def3b}
\end{align}

Define $\vartheta =\sum_{\alpha=1}^{m-1} \varepsilon_{2\alpha -1} \wedge\varepsilon_{2\alpha}$, then we have
\begin{equation}
\sigma_k =\frac{1}{k!} \vartheta^k +\frac{1}{(k-1)!}\vartheta^{k-1} \wedge\varepsilon_{2m-1}\wedge\varepsilon_{2m}.
\end{equation}

The second term of (\ref{eq:3Sasak-def3a}) is
\begin{equation}\label{eq:3Sasak-def4}
\begin{split}
-\sqrt{-1}\sqrt{2}h(\varepsilon_\beta)\varepsilon_{2m} \sigma_k & =-\frac{\sqrt{-1}2\sqrt{2}}{k!}\varepsilon_{2m}\wedge\bigl(h(\varepsilon_{\beta})\contr\vartheta^k \bigr) \\
 & = -\frac{\sqrt{-1}2\sqrt{2}}{(k-1)!}\varepsilon_{2m}\wedge\bigl(h(\varepsilon_{\beta})\contr\vartheta \bigr)\wedge\vartheta^{k-1} \\
 & = -\frac{\sqrt{-1}2\sqrt{2}}{(k-1)!}\varepsilon_{2m}\wedge \Phi_2 h(\varepsilon_{\beta})\wedge\vartheta^{k-1}.
\end{split}
\end{equation}
Note that every term of (\ref{eq:3Sasak-def4}) contains $\varepsilon_{2m}$ but does not contain $\varepsilon_{2m-1}$.
The terms of the first component of (\ref{eq:3Sasak-def3a}) which also contain $\varepsilon_{2m}$ but not
$\varepsilon_{2m-1}$ are
\begin{equation}\label{eq:3Sasak-def5}
\sum_{\alpha=1}^{2m-1}\nabla_{\ol{\alpha}}^T h_{\beta 2m-1}\varepsilon^{\ol{\alpha}}\varepsilon^{2m-1}\sigma_k.
\end{equation}

We simplify (\ref{eq:3Sasak-def5}) to get
\begin{equation}\label{eq:3Sasak-def6}
\begin{split}
\sum_{\alpha=1}^{2m-1}\nabla_{\ol{\alpha}}^T h_{\beta 2m-1}\varepsilon^{\ol{\alpha}}\varepsilon^{2m-1}\sigma_k & =
\sum_{\alpha=1}^{2m-1} -h(\varepsilon_{\beta}, \nabla^T_{\epsilon_{\ol{\alpha}}}\epsilon_{2m-1})\varepsilon^{\ol{\alpha}}\varepsilon^{2m-1}\sigma_k \\
& = \sqrt{-1}\sqrt{2} \sum_{\alpha=1}^{2m-1}h(\varepsilon_{\beta},\Phi_2 \epsilon_{\ol{\alpha}})\varepsilon^{\ol{\alpha}}\varepsilon^{2m-1}\sigma_k \\
& =-\sqrt{-1}\sqrt{2}\Phi_2 h(\varepsilon_{\beta})\varepsilon^{2m-1}\sigma_k \\
& =\frac{\sqrt{-1}2\sqrt{2}}{(k-1)!}\Phi_2 h(\varepsilon_{\beta})\wedge \vartheta^{k-1}\wedge\varepsilon_{2m}.
\end{split}
\end{equation}
Together the terms of (\ref{eq:3Sasak-def3a}) which contain $\varepsilon_{2m}$ but not
$\varepsilon_{2m-1}$ are
\begin{equation}\label{eq:3Sasak-def7}
-\frac{\sqrt{-1}4\sqrt{2}}{(k-1)!}\varepsilon_{2m}\wedge \Phi_2 h(\varepsilon_{\beta})\wedge\vartheta^{k-1}.
\end{equation}
We claim that (\ref{eq:3Sasak-def7}) is non-zero for $1\leq k\leq m-1$ when $\Phi_2 h(\varepsilon_{\beta})$ is non-zero.
But this follows because $\vartheta$ is a complex symplectic form on $\mathcal{D}$.  Thus $h(\varepsilon_{\beta})=0$.

A similar argument will be carried out with (\ref{eq:3Sasak-def3b}).
The second term of (\ref{eq:3Sasak-def3b}) is
\begin{equation}\label{eq:3Sasak-def8}
\begin{split}
\sqrt{-1}\sqrt{2} h(\varepsilon_{\ol{\beta}})\varepsilon_{\ol{2m}} \sigma_k & = \frac{\sqrt{-1}2}{(k-1)!}h(\varepsilon_{\ol{\beta}})\bigl(\vartheta^{k-1}\wedge\varepsilon_{2m-1}\bigr) \\
 & = \frac{\sqrt{-1}2\sqrt{2}}{(k-1)!}h(\varepsilon_{\ol{\beta}})\wedge\vartheta^{k-1}\wedge\varepsilon_{2m-1}.
\end{split}
\end{equation}
The terms of the first component of (\ref{eq:3Sasak-def3b}) which contain $\varepsilon_{2m-1}$ but not $\varepsilon_{2m}$
are
\begin{equation}
\sum_{\alpha=1}^{2m-1} \nabla_{\alpha}^T h_{\ol{\beta}\ol{2m-1}}\varepsilon^{\alpha}\varepsilon^{\ol{2m-1}}\sigma_k.
\end{equation}
We compute
\begin{equation}\label{eq:3Sasak-def9}
\begin{split}
\sum_{\alpha=1}^{2m-1} \nabla_{\alpha}^T h_{\ol{\beta}\ol{2m-1}}\varepsilon^{\alpha}\varepsilon^{\ol{2m-1}}\sigma_k & =
\sum_{\alpha=1}^{2m-1} -h(\varepsilon_{\ol{\beta}},\nabla_{\alpha}^T \varepsilon_{\ol{2m-1}})\varepsilon^{\alpha}\varepsilon^{\ol{2m-1}}\sigma_k \\
 & = -\sqrt{-1}\sqrt{2}\sum_{\alpha=1}^{2m-1}g(h(\varepsilon_{\ol{\beta}}),\Phi_2 \varepsilon_{\alpha})\varepsilon^{\alpha}\varepsilon^{\ol{2m-1}}\sigma_k \\
 & =\sqrt{-1}\sqrt{2}\Phi_2 h(\varepsilon_{\ol{\beta}})\varepsilon_{2m-1} \sigma_k \\
 & =\frac{\sqrt{-1}2\sqrt{2}}{k!}\varepsilon_{2m-1}\wedge\bigl(\Phi_2 h(\varepsilon_{\ol{\beta}})\contr\vartheta^k \bigr) \\
 & =\frac{\sqrt{-1}2\sqrt{2}}{(k-1)!}\varepsilon_{2m-1}\wedge\bigl(\Phi_2 h(\varepsilon_{\ol{\beta}})\contr\vartheta \bigr)\wedge\vartheta^{k-1} \\
 & =-\frac{\sqrt{-1}2\sqrt{2}}{(k-1)!}\varepsilon_{2m-1}\wedge h(\varepsilon_{\ol{\beta}})\wedge\vartheta^{k-1}
\end{split}
\end{equation}
Combining (\ref{eq:3Sasak-def8}) and (\ref{eq:3Sasak-def9}) give
\begin{equation}\label{eq:3Sasak-def10}
-\frac{\sqrt{-1}4\sqrt{2}}{(k-1)!}\varepsilon_{2m-1}\wedge h(\varepsilon_{\ol{\beta}})\wedge\vartheta^{k-1}.
\end{equation}

We have for $X\in\Gamma(D^{1,0})$ that the component of $\mathcal{L}(\alpha,\sigma_k)(X)$ containing $\varepsilon_{2m}$
but not $\varepsilon_{2m-1}$ is $-\frac{1}{2}$ of (\ref{eq:3Sasak-def7}).  Since these terms are linearly independent,
for $\sigma=\sum_{k=1}^{m-1}a_k \sigma_k$, $\mathcal{L}(\alpha,\sigma_k)(X)=0$ for all $X$ implies $h=0$.
\end{proof}

The proof involved determining the component of (\ref{eq:Kill-spin-inf}) with the spinor component containing precisely
one vector in $\Span_{\C}\{\varepsilon_{2m-1},\varepsilon_{2m}\}$.  This is given in (\ref{eq:3Sasak-def7}) and
(\ref{eq:3Sasak-def10}).  This component is preserved under changes of the frame used in the calculation.  This will be used
later in Section~\ref{subsec:3Sasak-def} where more details will be given.  It will be useful that this component is
\begin{equation}\label{eq:def-comp}
-\Phi_1 h(X)\xi_1 \cdot\sigma -h(X)\partial_r \cdot\sigma.
\end{equation}

\section{Integrable deformations of Killing spinors}\label{sec:int-def}

We consider the integrability of the infinitesimal Einstein deformations $h^\beta \in\EED(g)$ for $\beta\in\mathcal{H}^1_{\mathcal{A}}$
from the last section.  We will also consider the integrability of infinitesimal Killing spinor deformations.
This is essentially the problem of deforming Sasaki-Einstein metrics.  We give some sufficient conditions for integrating
these infinitesimal deformations.  A deeper sufficient condition for deforming Sasaki-Einstein metrics is K-polystability
(see~\cite{vanCoTip15}), but here we merely give some sufficiency results using analytic methods.

\subsection{Integrability on Sasaki-Einstein manifolds}

We state a result from~\cite{vanCo12} giving a sufficient condition for deforming Sasaki-Einstein structures.
Let $(M,g,\eta,\xi,\Phi)$ be a Sasaki-Einstein structure, and let $G\subseteq G'=\Aut(g,\eta,\xi,\Phi)$ be a
compact subgroup.  We consider $G$-equivariant deformations of the foliation $(\mathscr{F}_\xi,\ol{J})$.
We have the $G$-equivariant Dolbeault complex
\begin{equation}\label{eq:Dol-comp-G}
 0\rightarrow \mathcal{A}_G^{0} \overset{\ol{\partial}_b}{\longrightarrow}\mathcal{A}_G^{1} \overset{\ol{\partial}_b}{\longrightarrow}\mathcal{A}_G^{2}\rightarrow\cdots,
\end{equation}
with $\mathcal{A}_G^k =\Gamma(\Lambda^{0,k}_b\otimes\nu(\mathscr{F})^{1,0})^G$ the subspace of $G$-invariant sections.
Then $H^1(\mathcal{A}_G^\bullet)$ gives the first order deformations of $(\mathscr{F}_\xi,\ol{J})$ preserving the action
of $G$.  We saw in Proposition~\ref{prop:vers-def} that the versal deformation space $\mathcal{U}$ is smooth.  The space
of $G$-equivariant deformations $\mathcal{U}^G \subseteq\mathcal{U}$ is a submanifold with tangent space
$H^1(\mathcal{A}_G^\bullet)\subseteq H^1(\mathcal{A})$.  With respect to a fixed transversal K\"{a}hler structure
we have the $G$-invariant harmonic space $\mathcal{H}^1_{\mathcal{A},G}$ and
$H^1(\mathcal{A}_G^\bullet)\cong \mathcal{H}^1_{\mathcal{A},G}$.

If $(\mathscr{F}_\xi,\ol{J}_t)_{t\in\mathcal{V}}$ is a $G$-equivariant deformation, then one can show as in
Proposition~\ref{prop:vers-def-Sasaki} that there is a family of Sasakian structures $(g_t, \eta_t,\xi,\Phi_t),\ t\in\mathcal{V},$
with $G\subseteq \Aut(g_t,\eta_t,\xi,\Phi_t)$ where $\Phi_t$ induces the transversal complex structure $\ol{J}_t$.
Arguments using the implicit function theorem can show the following.
\begin{thm}[\cite{vanCo12}]\label{thm:def-Sasak-Einst}
Suppose $(M,g,\eta,\xi,\Phi)$ is a Sasaki-Einstein manifold.
Let $G\subseteq\Aut(g,\eta,\xi,\Phi)$ be a maximal torus, and let $(\mathscr{F}_\xi,\ol{J}_t)_{t\in\mathcal{V}}$ be
a $G$-equivariant deformation with $\mathcal{V}$ smooth.  Then after possibly shrinking $\mathcal{V}$, there is
a family $(g_t, \eta_t,\xi,\Phi_t),\ t\in\mathcal{V}$ of Sasaki-Einstein structures with
$(g_0, \eta_0,\xi,\Phi_0)=(g,\eta,\xi,\Phi)$ and with $\Phi_t$ inducing the transversal complex structure $\ol{J}_t$.
\end{thm}

This implies the following in terms of Killing spinors.
\begin{cor}\label{cor:def-Sasak-Einst}
Let $(M,g)$ be a spin Sasaki-Einstein manifold admitting the two defining Killing spinors $\sigma_j,\, j=0,1$, e.g. $M$ is
simply connected.  Then the infinitesimal Einstein deformations $h^\beta$, for $\beta\in\mathcal{H}^1_{\mathcal{A},G}$,
integrate to a family $g_t,\ t\in\mathcal{V}\subset\C^d,\ d=\dim_{\C}\mathcal{H}^1_{\mathcal{A},G}$, of Einstein
deformations preserving $\sigma_j,\, j=0,1$.

The components in $\EED(g)$ of $\{ v(g_t)\ |\ v\in T_0 \mathcal{V}\}$ are precisely the original infinitesimal Einstein
deformations $\{ h^\beta \ |\ \beta\in\mathcal{H}^1_{\mathcal{A},G}\}$.
\end{cor}
\begin{proof}
Just the last statement remains to be proved.
Consider the family $(g_t,\eta_t,\xi,\Phi_t),\ t\in\mathcal{V}$ of Proposition~\ref{prop:vers-def-Sasaki}.
Using the notation of Section~\ref{subsec:def-versdef} and differentiating in the direction of some $v\in T_0 \mathcal{V}$ we have
\begin{gather}
  \phi_{\alpha\beta} =0 \label{eq:def-kah1} \\
  \phi_{\alpha\ol{\beta}} =\sqrt{-1}h_{\alpha\ol{\beta}} \label{eq:def-kah2}\\
  h_{\alpha\beta} =\sqrt{-1}I_{\alpha\beta}, \label{eq:def-kah3}
\end{gather}
which follow from (\ref{eq:Sasak-def4}), (\ref{eq:Sasak-def2}) and (\ref{eq:Sasak-def1}) respectively.
In the proof of Proposition~\ref{prop:vers-def-Sasaki} the basic cohomology class $[\omega^T_t]$ is constant.
Thus $\phi$ is an exact $(1,1)$-form.  We may replace $\eta_t$ with $\eta_t +d^c \psi_t$, so that using the same notation we have
$\frac{1}{2}d\dot{\eta}_t =\phi =0$.

The possible contact forms for a fixed Reeb vector field $\xi$ and transversal complex structure $\ol{J}_t$ are
$\eta_t +d^c \psi_t +d\theta_t$ for basic functions $\psi_t, \theta_t \in C_b^\infty(M)$.  See~\cite[Lemma 2.2.3]{vanCo12}, where
we also use that $\Ric^T >0$, which implies that the basic cohomology $H^1_b =H^1(M,\R)=\{0\}$.  And $d\theta_t$ is given by a gauge transformation $\exp(\theta_t \xi)^* \eta_t$, which fixes basic tensors.  Therefore, by adding a factor of $d\theta_t$ to $\eta_t$, we may arrange that $\dot{\eta}_t =0$.
We assume that the family $(g_t,\eta_t,\xi,\Phi_t),\ t\in\mathcal{V}$ is chosen so that $\dot{\eta}_t =0$ at $t=0$.
Thus the only component of $\dot{g}_t$ at $t=0$ is $h_{\alpha\beta}=\sqrt{-1}I_{\alpha\beta} \in\EED(g)$.

Recall that if $\psi \in C_b^\infty(M)$ is sufficiently small there is a Sasakian structure
$(g_{t,\psi}, \eta_{t,\psi},\xi,\Phi_{t,\psi})$ with contact form
$\eta_{t,\psi} =\eta_t +d^c \psi$ and transversal complex structure $\ol{J}_t$.  The metric is
\[ g_{t,\psi} =\frac{1}{2}d\eta_{t,\psi}(\cdot,\ol{J}_t \cdot) +\eta_{t,\psi} \otimes\eta_{t,\psi},\]
and $\Phi_{t,\psi}$ is the lift of $\ol{J}_t$ to $\ker\eta_{t,\psi}$.

Theorem~\ref{thm:def-Sasak-Einst} is proved by using the implicit function theorem to find
$\psi_t \in C_b^\infty(M),\ t\in\mathcal{V}$, so that the Sasakian structure $(g_{t,\psi}, \eta_{t,\psi},\xi,\Phi_{t,\psi})$
has scalar curvature $s_{t,\psi_t} =0$.  We review enough of the proof of Theorem~\ref{thm:def-Sasak-Einst} to prove
the corollary.  For more details see~\cite{vanCo12}.

We consider the $G$-invariant Sobolev space $L^2_{k+4,G}(M),\ k>m$, of $k+4$ times weakly differentiable functions.
For $\psi\in L^2_{k+4,G}(M)$ small we have the Sasakian structure with metric $g_{t,\psi}$ as above.
We have the space of holomorphy potentials $\mathcal{H}^{\mathfrak{g}}_{t,\psi}$ for this metric where $\mathfrak{g}$ is
the Lie algebra of $G$ (cf.~\cite{vanCo12}).
Using the metric $g_{t,\psi}$ to define the $L^2$ inner product on $L^2_{k,G}(M)$ we have the orthogonal
decomposition
\[ L^2_{k,G}(M)=\sqrt{-1}\mathcal{H}^{\mathfrak{g}}_{t,\psi} \oplus W_{k,t,\psi}, \]
and the projections
\[ \pi^G_{t,\psi} :L^2_{k,G}(M)\rightarrow\sqrt{-1}\mathcal{H}^{\mathfrak{g}}_{t,\psi},\text{  and  }
\pi^W_{t,\psi} :L^2_{k,G}(M)\rightarrow W_{k,t,\psi}. \]
The reduced scalar curvature of $g_{t,\psi}$ is given by
\begin{equation}
s^G_{t,\psi} =\pi^W_{t,\psi}(s_{t,\psi})=(\mathbb{1} -\pi^G_{t,\psi})(s_{t,\psi}).
\end{equation}

Let $U\subset\mathcal{V}\times L^2_{k+4,G}(M)$ be a neighborhood of $(0,0)$ so that for $(t,\psi)\in U$,
$(g_{t,\psi}, \eta_{t,\psi},\xi,\Phi_{t,\psi})$ is well defined.
For $\mathcal{U}=U \cap\bigl(\mathcal{V}\times W_{k+4,0} \bigr)$ we define a map
\begin{equation}\label{eq:def-map}
\begin{array}{rccc}
\mathcal{S}: & \mathcal{U} & \rightarrow &  W_{k,0}\\
			 & (t,\psi) & \mapsto & \pi^W_{0}(s^G_{t,\psi}).
\end{array}
\end{equation}
The derivative of (\ref{eq:def-map}) is
\begin{equation}\label{eq:def-map-der}
 d\mathcal{S} : W_{k+4,0} \rightarrow W_{k,0},
\end{equation}
with $d\mathcal{S}(\dot{\psi}) = -2\mathbb{L}_g \dot{\psi}$.  Here $\mathbb{L}_g $ is the self-adjoint operator
\begin{equation*}
\mathbb{L}_g \psi =\frac{1}{2}\Delta_b^2 \psi +\frac{1}{2}(\Ric^T, dd^c \psi)+ \frac{1}{2}(d\psi,ds_g).
\end{equation*}
As proved in~\cite[Cor. 4.2.5]{vanCo12} there is a family $\psi_t,\ t\in\mathcal{U},$ with
\begin{equation}\label{eq:def-map2}
\mathcal{S}(t,\psi_t) =\pi^W_{0}(s^G_{t,\psi_t})=0.
\end{equation}
Since $\dot{g}_t \in\EED(g)$ it is easy to check that $\frac{d}{dt} s^G_{t,0}=0$ at $t=0$.  Then differentiating
(\ref{eq:def-map2}) at $t=0$ gives $-2\mathbb{L}_g \dot{\psi}_t =0$.  But (\ref{eq:def-map-der}) is an isomorphism,
so $\dot{\psi}_t =0$ at $t=0$.  Therefore at $t=0$ we have $\dot{g}_{t,\psi_t}=\dot{g}_t$ which is
$h_{\alpha\beta}=\sqrt{-1}I_{\alpha\beta} \in\EED(g)$.
\end{proof}

We will give an application of Theorem~\ref{cor:def-Sasak-Einst} in Section~\ref{subsec:examples}.

\subsection{Integrability on 3-Sasakian manifolds}

We can prove integrability of many of the transversal infinitesimal deformations on a 3-Sasakian manifold.
The infinitesimal deformations of the real subspace $\re \mathcal{H}^1_{\mathcal{A}}(\xi)\subset\mathcal{H}^1_{\mathcal{A}}(\xi)$ with respect to the real structure $\varsigma: \mathcal{H}^1_{\mathcal{A}}(\xi)\rightarrow \mathcal{H}^1_{\mathcal{A}}(\xi)$ induced
by the anti-holomorphic real structure $\varsigma:Z\rightarrow Z$ integrate to Einstein deformations preserving the
existence of precisely two Killing spinors.

\begin{thm}\label{thm:3Sasak-int-def}
Let $(M,g),\ \dim M=4m-1,$ be a 3-Sasakian manifold, and denote by $\sigma_i,\ i=0,\ldots, m$ the Killing spinors associated to the 3-Sasakian structure.  Then the infinitesimal Einstein deformations $h^\beta$ of $g$ for $\beta\in\re \mathcal{H}^1_{\mathcal{A}}(\xi)$
in Proposition~\ref{prop:inf-def-Sasaki} integrate to a family $g_t,\ t\in\mathcal{N}\subset\R^d,\ d=\dim_{\C} \mathcal{H}^1_{\mathcal{A}},$ of Einstein deformations of $g$ preserving $\sigma_0$ and $\sigma_m$ but not the remaining.
The components in $\EED(g)$ of $\{v(g_t)\ |\ v\in T_0 \mathcal{N}\}$ are precisely the original infinitesimal Einstein deformations
$\{h^\beta\ |\ \beta\in\mathcal{H}^1_{\mathcal{A}}(\xi)\}$.
\end{thm}

\begin{cor}
Let $(M,g),\ \dim M=4m-1,$ be a 3-Sasakian manifold with $d =\dim_{\C}H^1(\mathcal{A}^\bullet)$.  Then $g$ has a $d$-dimensional family of non-trivial deformations, $\{g_t \ |\ t\in\mathcal{N}\subset\R^d \}$, where $g_t,\ t\neq 0$, has a compatible Sasaki-Einstein structure but no 3-Sasakian structure.
\end{cor}

Recall that the quotient of $M$, $\dim M =4m+3$, by the action of $\Sp(1)$-action generated by $\{\xi_1,\xi_2,\xi_3 \}$
is a quaternion-K\"{a}hler orbifold $(\hat{M},\hat{g}),\ \dim\hat{M}=4m$.  If $m\geq 2$, this means
there is a three dimensional bundle $\mathcal{J}\subset\End(TM)$ which is locally spanned by almost complex structures
$\hat{J}_i,\ i=1,2,3$ satisfying the quaternionic identities which is preserved by the Levi-Civita connection of $\hat{g}$.
This is equivalent to the existence of a 1-integrable $\Sp(m)\Sp(1)$-structure on $\hat{M}$.
The O'Neill formulas of the submersion $\pi: M\rightarrow\hat{M}$ show that $(\hat{M},\hat{g})$ is Einstein with constant
$\lambda =4m+8$.  If $m=1$, every oriented manifold
satisfies this with $\mathcal{J} =\Lambda^2_+$.  A 4-dimensional quaternion-K\"{a}hler orbifold $(\hat{M},\hat{g})$ is defined to be oriented and satisfy $W^+_{g} \equiv 0$ and $\Ric_g =\lambda g$.

We will consider a weaker condition, that of a quaternionic structure (cf.~\cite{Sal86}).
\begin{defn}\label{defn:quatern-str}
A \emph{quaternionic structure} on $\hat{M}$, of dimension $4m,\ m\geq 2$, is a three dimensional subbundle
$\mathcal{J}\subset\End(T\hat{M})$ which is locally spanned by almost complex structures
$\hat{J}_i,\ i=1,2,3$ satisfying the quaternionic identities and preserved by a torsion-free connection on $T\hat{M}$.
This is equivalent to the existence of a 1-integrable $\GL(m,\Ha)\Sp(1)$-structure.

If $m=1$, then a quaternionic structure is defined to be a conformal class $[g]$ with an orientation on $\hat{M}$ satisfying
$W^+_{[g]} \equiv 0$.
\end{defn}

Part of the interest in quaternionic manifolds is due to an attractive twistor correspondence~\cite{Sal84}.
If $(\hat{M},\mathcal{J})$ is a
$4m$-dimensional quaternionic manifold, then the \emph{twistor space} is $Z=\mathbb{P}(\mathbf{E})$ where $\mathbf{E}$ is the locally defined complex 2-dimensional bundle associated to the complex 2-dimensional representation of the $\Sp(1)$-factor of
$\GL(m,\Ha)\Sp(1)$.  Then $Z$ is a $2m+1$-dimensional complex manifold with
a family of twistor lines $\cps^1$ with normal bundle $\mathcal{O}_{\cps^1}(1)^{\oplus 2m}$ and an anti-holomorphic involution
$\varsigma:Z\rightarrow Z$ preserving the \emph{real} twistor lines.  Conversely, if $Z$ is a $2m+1$-dimensional complex manifold with
a family of twistor lines $\cps^1$ with normal bundle $\mathcal{O}_{\cps^1}(1)^{\oplus 2m}$ and an anti-holomorphic involution
$\sigma:Z\rightarrow Z$, then a connected component of real twistor lines is a $4m$-dimensional manifold with a quaternionic structure.
Since the twistor correspondence is natural, if $(M,\mathcal{J})$ is a quaternionic orbifold we may define the twistor space over each uniformizing chart as for manifolds and quotient by the orbifold group.

We say that a diffeomorphism of a quaternionic manifold(orbifold) $F:\hat{M}\rightarrow\hat{M}$ is a quaternionic automorphism if the derivative of $F$ preserves the bundle $\mathcal{J}$, or equivalently preserves the $\GL(m,\Ha)\Sp(1)$-structure.
The following is essentially different proof of a result of C. LeBrun~\cite[Corollary C]{LeB95}, but we need to consider the case in which $(\hat{M},\hat{g})$ is an orbifold.
\begin{lem}\label{lem:quater-aut}
Let $(\hat{M},\hat{g})$ be a quaternion-K\"{a}hler manifold or orbifold whose associated 3-Sasakian space $M$ is smooth.
If $(\hat{M},\hat{g})$ admits a quaternionic automorphism which is not an isometry,
then $(\hat{M},\hat{g})$ is locally isometric to $\qps^m$ with the symmetric metric.  Thus
$(\hat{M},\hat{g})\overset{\operatorname{isom}}{\cong}\Gamma\backslash\qps^m ,\ \Gamma\subset\Sp(m+1)$.
\end{lem}
\begin{proof}
Let $M\rightarrow\hat{M}$ be the $\Sp(1)$ or $\SO(3)$ orbifold bundle with $M$ the 3-Sasakian space associated to $\hat{M}$.
Suppose there is such a quaternionic automorphism $F:\hat{M}\rightarrow\hat{M}$, then $F$ lifts to a diffeomorphism
$\ol{F}:M\rightarrow M$ which maps each $\xi_i,\ i=1,2,3$ to itself and preserves the complex structure on the transverse space $Z$.
The complex contact form $\theta$ of $Z$ lifts to $\eta^c=\eta_2 +\sqrt{-1}\eta_3$.
Since $F:\hat{M}\rightarrow\hat{M}$ is an isometry if and only if the biholomorphism induced on $Z$ is complex contact~\cite{LeB89,NitTak87}, $\hat{\eta}=\ol{F}^* \eta^c \neq\eta^c$.  And $C(M)$ has two holomorphic symplectic forms
$\varpi =d(r^2 \eta^c)$ and $\hat{\varpi} =d(r^2 \hat{\eta})$.  If $\ol{\nabla}$ is the Levi-Civita connection of $(C(M),\ol{g})$,
then $\ol{\nabla}\varpi =0$. Note that both $\varpi$ and $\hat{\varpi}$ are of order 2 with respect to the Euler vector field $r\partial_r$.  Since $\ol{\nabla}_{\partial_r} \partial_r =0$ and $\ol{\nabla}_{r\partial_r} X=X$ for a vector field $X$ on $M$
viewed as a vector field on $C(M)$, it is easy to check that $\ol{\nabla}_{\partial_r} \hat{\varpi}=0$.

We have the following formula on a K\"{a}hler-Einstein manifold with Einstein constant $\lambda$
\begin{equation}
\ol{\nabla}^{\beta}\ol{\nabla}_{\beta}\hat{\varpi}_{\alpha_1 \alpha_2} =\ol{\nabla}^{\ol{\beta}}\ol{\nabla}_{\ol{\beta}} \hat{\varpi}_{\alpha_1 \alpha_2} +
2\lambda \hat{\varpi}_{\alpha_1 \alpha_2}.
\end{equation}
Since $\lambda=0$ and $\hat{\varpi}$ is holomorphic, we have $\ol{\nabla}^{\beta}\ol{\nabla}_{\beta} \hat{\varpi}_{\alpha_1 \alpha_2} =\ol{\nabla}^{\ol{\beta}}\ol{\nabla}_{\ol{\beta}} \hat{\varpi}_{\alpha_1 \alpha_2} =0$.
Consider $TC(M)|_M$ as an Hermitian vector bundle on $M$ and denote by $\nabla$ the connection $\ol{\nabla}$ restricted to $M$.
Then $\nabla^*\nabla\hat{\varpi} =\ol{\nabla}^*\ol{\nabla}\hat{\varpi} =0$ and
\begin{equation*}
\begin{split}
0 & =\int_M \langle \nabla^*\nabla\hat{\varpi} ,\hat{\varpi} \rangle\, \mu_{g} \\
  & =\int_M \langle \nabla\hat{\varpi} ,\nabla\hat{\varpi} \rangle\, \mu_g.
\end{split}
\end{equation*}
Therefore $\ol{\nabla}\hat{\varpi} =0$.  So the holonomy of $(C(M),\ol{g})$ stabilizes two linearly independent $(2,0)$-forms
of maximal rank, and the holonomy of the universal cover $\tilde{C(M)}$ is reducible.  It follows from~\cite[Prop. 3.1]{Gal79} that $\tilde{C(M)}$ is flat.  Thus $M$ is isometric to a space form $\Gamma\backslash\operatorname{S}^{4m+3}$.
\end{proof}

\begin{proof}[Proof of Theorem]
Fixing a $\xi\in\SR$ we have the foliation $(\mathscr{F}_\xi,\ol{J})$ whose transversal space is the twistor space $Z$.
There is a subspace $\mathcal{N}\subset\mathcal{U}\subset H^1(\mathcal{A}^\bullet)$ of the versal deformation space of
$(\mathscr{F}_\xi,\ol{J})$ of \emph{real deformations}.  These are the deformations $\ol{J}_t$ for which
$\varsigma(\ol{J}_t)=-\ol{J}_t$.  By straightforward averaging one can choose the family of compatible Sasakian
structures in Proposition~\ref{prop:vers-def-Sasaki} $(g_t,\eta_t,\xi,\Phi_t)$ to satisfy
\begin{equation}\label{eq:sigma-inv}
\varsigma^* g_t =g_t, \quad \varsigma^* \eta_t =-\eta_t, \quad \varsigma_* \xi=-\xi,\quad \varsigma^*\Phi_t =-\Phi_t,
\end{equation}
for $t\in\mathcal{N}$.  In particular, we also have $\varsigma^*\omega^T =-\omega^T$.
For $t\in\mathcal{N}$ with respect to $(g_t,\eta_t,\xi,\Phi_t)$ we have
$\re H^1(\mathcal{A}^\bullet) =\re\mathcal{H}^1_{\mathcal{A}}(\xi)$ for the tangent space to $\mathcal{N}$ at $0$.
Therefore $(\mathscr{F}_\xi,\ol{J}_t)=(Z,\ol{J}_t)$ has a K\"{a}hler structure $\omega^T_t$, with
$\omega^T_t \in\frac{\pi}{2m} c_1(\mathscr{F}_\xi ,\ol{J}_0)$ depending smoothly on $t\in\mathcal{N}$
and $\Ricci(\omega^T_0)=4m\omega^T_0$.  Since the leaf space is an orbifold we will denote the transversal K\"{a}hler
space by $(Z,\ol{J}_t ,\omega_t)$.

Let $\mathfrak{g}$ be the Lie algebra of quaternionic automorphisms of $(\hat{M},\hat{g})$.
By the twistor correspondence, $\mathfrak{g}\cong\{X\in\hol(Z,J_0)| \varsigma_* X =X\}$.
Since $\mathfrak{g}$ is a real form of $\hol(Z,\ol{J}_0)$, $\mathfrak{g}\otimes\C =\hol(Z,\ol{J}_0)$.
By Lemma~\ref{lem:quater-aut} $\mathfrak{g}\subseteq\isom(\hat{M},\hat{g},\mathcal{J})$.  Thus
$\mathfrak{g}\subseteq\isom(Z,\omega_0,\ol{J}_0)$.
Since $(Z,\omega_0,\ol{J}_0)$ is K\"{a}hler-Einstein the results of Y. Matsushima~\cite{Mat57} show that
$\isom(Z,\omega_0,\ol{J}_0)\subset\mathfrak{g}\otimes\C$ is a real form, so $\mathfrak{g}=\isom(Z,\omega_0,\ol{J}_0)$.

Recall that $f\in C^\infty(Z,\C)$ is a \emph{holomorphy potential} if
\[\partial^{\#}f :=(\ol{\partial}f)^{\#}=\sum_{i,j}\frac{\partial f}{\partial\ol{z}^{\ol{k}}} g^{\ol{k}j}\frac{\partial}{\partial z^j}\]
is holomorphic.  We define the space of normalized holomorphy potential functions,
\begin{equation}
\mathcal{H}_g :=\{f\in C^\infty(Z,\C)\ |\ f\text{ is Hamiltonian and }\int f\,\mu_g =0 \}.
\end{equation}
Suppose $W\in\Gamma(T^{1,0}Z)$ is holomorphic with $\re W=X \in\mathfrak{g}=\{Y\in\hol(Z,J_0)\ |\ \varsigma_* Y =Y\}$, so
$\mathcal{L}_X \omega =0$.  And let $f_W \in C^\infty(Z)$ be a symplectic Hamiltonian, with $\int_Z f_W \mu_g =0$, that is
\begin{equation}\label{eq:symp-Ham}
X\contr\omega =df_W
\end{equation}
Then
\begin{equation*}
\partial^{\#} f_W =\frac{1}{2} (df_W +\sqrt{-1}J^*df_W)^{\#} =\frac{1}{2}(JX+\sqrt{-1}X) =\frac{\sqrt{-1}}{2}W
\end{equation*}
From (\ref{eq:sigma-inv}) and (\ref{eq:symp-Ham}) we have $\varsigma^* df_W =-df_W$, and $\int_Z f_W \,\mu_g =0$ implies that
$\varsigma^* f_W =-f_W$.  Since $\mathcal{H}_g$ is the complexification of the real functions $f_W$ considered,
we have that $\varsigma^*f =-f$ for all $f\in\mathcal{H}_g$.

There are $F_t \in C^\infty(Z)$ depending smoothly on $t\in\mathcal{N}$ with
\begin{equation}\label{eq:Ricci-pot}
\sqrt{-1}\partial_t \ol{\partial}_t F_t =\Ricci(\omega_t) -4m\omega_t.
\end{equation}
Since $F_t $ is defined up to a constant, $\varsigma^* F_t =F_t +c_t$ for $t\in\re\mathcal{N}$.
But $\int (F_t -\varsigma^* F_t)\mu_{g_t} =0$, so $\varsigma^* F_t =F_t$.

Define $C^{k,\alpha}(Z)_{\operatorname{sym}}$ to be the H\"{o}lder space of functions $f$ with $\varsigma^*f =f$.
The Monge-Amp\`ere equation
\begin{equation}\label{eq:M-A-eq}
\Psi(\varphi_t,t) =\log\Bigl(\frac{(\omega_t +\sqrt{-1}\partial_t\ol{\partial}_t \varphi_t)^{2m-1}}{\omega_t^{2m-1}} \Bigr) +4m\varphi_t =F_t,
\end{equation}
is $\varsigma$-invariant for $t\in\re\mathcal{N}$, and $\Psi$ defines a smooth map
\begin{equation}\label{eq:M-A-map}
\Psi :C^{k+2,\alpha}(Z)_{\operatorname{sym}} \times\re\mathcal{N} \rightarrow C^{k,\alpha}(Z)_{\operatorname{sym}}
\end{equation}
The differential of (\ref{eq:M-A-map}) is
\begin{equation}
D_{\varphi}\Psi(\dot{\varphi}) =\bigl(-\Delta_{\ol{\partial}} +4m\bigr)\dot{\varphi}.
\end{equation}
But it is a result of Y. Matsushima~\cite{Mat57} that $\mathcal{H}_g =\ker(\Delta_{\ol{\partial}}-\lambda)$, where $\lambda =4m$ is
the Einstein constant.
Thus $D_{\varphi}\Psi: C^{k+2,\alpha}(Z)_{\operatorname{sym}}\rightarrow C^{k,\alpha}(Z)_{\operatorname{sym}}$ is an isomorphism.
By the implicit function theorem, after possibly replacing $\mathcal{N}$ by a smaller neighborhood of $0$,
for $t\in\mathcal{N}$ there is a $\varphi_t \in C^{k+2,\alpha}(Z)_{\operatorname{sym}}$ with $\Psi(\varphi_t) =F_t$, and
\begin{equation}\label{eq:KE-metric}
\omega'_t =\omega_t +\sqrt{-1}\partial_t \ol{\partial}_t \varphi_t
\end{equation}
is K\"{a}hler-Einstein.  The well-known regularity results show that $\varphi_t \in C^{\infty}(Z)_{\operatorname{sym}}$.

Let $\pi: M_t \rightarrow Z_t$ be the $\U(1)$-bundle associated to either $\mathbf{K}_{Z_t}^{\frac{1}{m}}$ or $\mathbf{K}_{Z_t}^{-\frac{1}{2m}}$, depending on whether $(M,g)$ fibers over $(\hat{M},\hat{g})$ with generic $\SO(3)$ or $\Sp(1)$ fibers.  Choose the connection form on $M_t$ to be $\eta_t' =\eta_t +d_t^c \varphi_t$.
Then from (\ref{eq:KE-metric}) one has $\frac{1}{2}d\eta_t =\omega'_t$.
We get a Sasaki-Einstein structure $(g'_t,\eta'_t,\xi,\Phi'_t)$ on $M_t$ where
\begin{equation}
\tilde{g}'_t =\omega'_t(\cdot,\Phi'_t \cdot) +\eta'_t \otimes\eta'_t,
\end{equation}
and $\Phi'_t$ is the lift of $\ol{J}_t$ to $\ker\eta'_t$.

By Theorem~\ref{thm:PedPo} for small $t\in\mathcal{N}$, $(M, g'_t)$ has no compatible 3-Sasakian structure.

It remains to prove that the components in $\EED(g)$ of $\{v(g_t)\ |\ v\in T_0 \mathcal{N}\}$ are precisely the original infinitesimal Einstein deformations $\{h^\beta\ |\ \beta\in\mathcal{H}^1_{\mathcal{A}}(\xi)\}$.
Consider the family $(g_t,\eta_t,\xi,\Phi_t),\ t\in\mathcal{N}$ of Proposition~\ref{prop:vers-def-Sasaki}.
Using the notation of Section~\ref{subsec:def-versdef} and differentiating in the direction of some $v\in T_0 \mathcal{N}$ we have
\begin{gather}
  \phi_{\alpha\beta} =0 \label{eq:def-comp1} \\
  \phi_{\alpha\ol{\beta}} =\sqrt{-1}h_{\alpha\ol{\beta}} \label{eq:def-comp2}\\
  h_{\alpha\beta} =\sqrt{-1}I_{\alpha\beta}, \label{eq:def-comp3}
\end{gather}
which follow from (\ref{eq:Sasak-def4}), (\ref{eq:Sasak-def2}) and (\ref{eq:Sasak-def1}) respectively.
In the proof of Proposition~\ref{prop:vers-def-Sasaki} the basic cohomology class $[\omega^T_t]$ is constant.
Thus $\phi$ is an exact $(1,1)$-form.  We may replace $\eta_t$ with $\eta_t +d^c \psi_t$, so that using the same notation we have
$\frac{1}{2}d\dot{\eta}_t =\phi =0$.

The possible contact forms for a fixed Reeb vector field $\xi$ and transversal complex structure $\ol{J}_t$ are
$\eta_t +d^c \psi_t +d\theta_t$ for basic functions $\psi_t, \theta_t \in C_b^\infty(M)$.  See~\cite[Lemma 2.2.3]{vanCo12}, where
we also use that $\Ric^T >0$,  which implies that the basic cohomology $H^1_b =H^1(M,\R)=\{0\}$.  And $d\theta_t$ is given by a gauge transformation $\exp(\theta_t \xi)^* \eta_t$, which fixes basic tensors.  Therefore, by adding a factor of $d\theta_t$ to $\eta_t$, we may arrange that $\dot{\eta}_t =0$.

We suppose now that we have chosen $(g_t,\eta_t,\xi,\Phi_t),\ t\in\mathcal{N}$ as such.
Thus the only component of $h$ is $h_{\alpha\beta} =\sqrt{-1}I_{\alpha\beta}$, which is a transversal
infinitesimal Einstein deformation.  Differentiating (\ref{eq:Ricci-pot}) gives
\[ \sqrt{-1}\partial_b \ol{\partial}_b \dot{F}_t =0. \]
Then differentiating (\ref{eq:M-A-eq}) with respect to $t$ gives
\[ \bigl(-\Delta_{\ol{\partial}} +4m\bigr)\dot{\varphi}_t =0, \]
and it follows that $\dot{\varphi}_t =0$ at $t=0$.
Therefore $(g'_t,\eta'_t,\xi,\Phi'_t)$ gives the same first order Einstein deformation at $t=0$ as $(g_t,\eta_t,\xi,\Phi_t)$
which is $h_{\alpha\beta} =\sqrt{-1}I_{\alpha\beta}$.
\end{proof}

\section{Deformations on a 3-Sasakian manifold}

\subsection{Space of Deformations on a 3-Sasakian manifold}\label{subsec:3Sasak-def}

The space of Einstein deformations on a 3-Sasakian manifold constructed in Section~\ref{sec:S-E-KS-def} has an interesting
structure.  Suppose $(M,g)$ has a 3-Sasakian structure with Reeb vector fields $\xi_1,\xi_2,\xi_3$ satisfying
$[\xi_i,\xi_j]=-2\varepsilon^{ijk}\xi_k$ and space of Reeb fields $\SR$.

For $\xi\in\SR$ and $\beta\in\mathcal{H}^1_{\mathcal{A}}(\xi)$ we define $h^{\beta,\xi} \in\EED(g)$, where
$h^{\beta,\xi} (X,Y)=g^T(\ol{J}\beta X, Y)$ where we distinguish the particular Reeb vector field.
We have the following space of infinitesimal Einstein deformations
\begin{equation}\label{eq:3Sasak-def-space}
\EDS(g):= \sum_{\xi\in\SR} \{ h^{\beta,\xi}\ |\ \beta\in\mathcal{H}^1_{\mathcal{A}}(\xi)\} \subseteq\EED(g)
\end{equation}

We have a left action of $\Sp(1)$ on $(M,g)$ generated by $\xi_1,\xi_2,\xi_3$.
Since $\Sp(1)$ acts by isometries and on the space of Sasakian structures $\SR$, it acts on $\EDS(g)$, and
all the subspaces $\mathcal{H}^1_{\mathcal{A}}(\xi), \xi\in\SR$, are isomorphic.
The subspace $\mathcal{H}^1_{\mathcal{A}}(\xi_1)$ is preserved by $\xi_1$, so by elementary representation theory
\[ \dim_{\R} \EDS(g) =2\dim_{\C} \EDS(g) \geq 6\dim_{\C}\mathcal{H}^1_{\mathcal{A}}(\xi_1).\]

This $\Sp(1)$-action acts on $(C(M),J_1,J_2,J_3)$ by
quaternionic automorphisms.  That is, it preserves the bundle of quaternionic frames $L_{\Sp(m)Sp(1)}(C(M))$.
This lifts, via the spin structure to an action on $\tilde{L}_{\Sp(m)Sp(1)}(C(M))\subset L_{\Spin(4m)}(C(M))$ if m is even
or $\tilde{L}_{\Sp(m)\times Sp(1)}(C(M))\subset L_{\Spin(4m)}(C(M))$ if m is odd.
The Killing spinors are contained in the $\gamma_m$ factor of $\mathbb{S}^+_{4m}$  of (\ref{eq:spin-rep-quat}).
Thus $\Sp(1)$ acts on the Killing spinors via the representation of $\Sp(1)=\SU(2)$ on $\gamma_m =S^2(\mu_2)$.

We will consider a principal subbundle $E\subset L_{\Sp(m)Sp(1)}(C(M))$ with structure group
$\bigl(\Sp(m-1)\times\Sp(1)\bigr)\Sp(1)$ generated by all the local frames considered in the proof of
Proposition~\ref{prop:inf-def-3Sasaki}.  This subbundle is invariant under the isometric $\Sp(1)$-action.
In order to determine the $\Sp(1)$ action on spinors we consider the spin bundle
\[\Sigma = \tilde{E}\times_{\bigl(\Sp(m-1)\times\Sp(1)\bigr)\Sp(1)} \mathbb{S}^+_{4m}. \]
Importantly, the subspace of spinors, considered in the proof of Proposition~\ref{prop:inf-def-3Sasaki}, with precisely
one vector in $\Span_{\C}\{\varepsilon_{2m-1},\varepsilon_{2m}\}$ is preserved by $\bigl(\Sp(m-1)\times\Sp(1)\bigr)\Sp(1)$.

The $\Sp(1)$ action on $E$ is easily computed.  Given $a\in\Sp(1)$ and $u\in E$, write $a_* u=u\psi(a)$, then
\[ \psi(a) =\bigl((\cdot,kak^{-1}),a) \in \bigl(\Sp(m-1)\times\Sp(1)\bigr)\Sp(1)\]
is the factor acting non-trivially on the component of spinors with one vector in $\Span_{\C}\{\varepsilon_{2m-1},\varepsilon_{2m}\}$.
It will be useful that the spin bundle has the decomposition (\ref{eq:spin-rep-quat}) with the $\Sp(1)$-action
acting on the $\gamma_m, \gamma_{m-2},\ldots $ factors in the usual way with $\gamma_m$ being the space of Killing spinors.

We will need a lemma in the proofs of the main theorems.
\begin{lem}\label{lem:ind-def}
Suppose $\xi,\xi' \in\SR$.  If $\xi \neq\xi'$ and $\xi \neq -\xi'$, then
\[ \{ h^{\beta,\xi}\ |\ \beta\in\mathcal{H}^1_{\mathcal{A}}(\xi)\}\cap\{ h^{\beta,\xi'}\ |\ \beta\in\mathcal{H}^1_{\mathcal{A}}(\xi')\} =\{0\}. \]

Suppose that $m=2$, $\xi_1,\xi_2,\xi_3 \in\SR$ are linearly independent, and $\beta_i \in\mathcal{H}^1_{\mathcal{A}}(\xi_i), i=1,2,3$ are non-zero.  Then
\[ h^{\beta_1,\xi_1} + h^{\beta_2,\xi_2} + h^{\beta_3,\xi_3} \neq 0. \]
\end{lem}
\begin{proof}
Let $\sigma_k ,k=0,\ldots,m$ be the Killing spinors as in the proof of Proposition~\ref{prop:inf-def-3Sasaki} which
span the representation $\gamma_m$ of $\Sp(1)$.  More precisely, $\gamma_m \cong S^2(\C^2)$ where we identify
$\Sp(1)\cong\SU(2)$.  And under this identification each $\sigma_k$ is identified with $\binom{m}{k}e_1^k e_2^{m-k}$ where
$e_1,e_2$ are the standard basis of $\C^2$.  By acting by $\Sp(1)$ we may suppose that $\xi$ is $\xi_1$.

By Proposition~\ref{prop:inf-def-3Sasaki} the elements
$h^{\beta,\xi}$ preserve the spinors corresponding to the span of $e_1^m$ and $e_2^m$ but not the remaining.
Let $g\in\SU(2)$ be such that $g\xi =\xi'$.  Then the elements $h^{\beta,\xi'}$ preserve precisely the spinors
$g(e_1^m)$ and $g(e_2^m)$.  This is the same subspace as that spanned by $e_1^m$ and $e_2^m$ if and only if $g$ is in the
subgroup generated by the elements
\[ \begin{bmatrix} u & 0\\ 0 & \ol{u} \end{bmatrix},\ \text{such that}\ |u|=1,\ \text{ and }\ J=\begin{bmatrix} 0 & 1\\ -1 & 0 \end{bmatrix}. \]
This is precisely the subgroup fixing $\xi_1\in\rps^2$.

For the second part recall that $\gamma_2$ is a real representation, with real Killing spinors
\begin{equation}\label{eq:real-KS}
\begin{aligned}
\varsigma_0 & = 1 +\varepsilon_1 \wedge\varepsilon_2 \wedge\varepsilon_3 \wedge\varepsilon_4 \\
\varsigma_1 & = i\varepsilon_1 \wedge\varepsilon_2 +i\varepsilon_3 \wedge\varepsilon_4 \\
\varsigma_2 & = i-i\varepsilon_1 \wedge\varepsilon_2 \wedge\varepsilon_3 \wedge\varepsilon_4. \\
\end{aligned}
\end{equation}
Again we may assume that $\xi_1$ is the standard Reeb vector field, thus $h^{\beta_1,\xi_1}$ preserves $\sigma_0$ and $\sigma_2$.
Suppose $\xi_2 =a\xi_1$ and $\xi_3 =b\xi_1$ where $a,b\in\Sp(1)$.  By assumption
$\Span_{\R}\{\sigma_0,\sigma_2\}\cap\Span_{\R}\{a\sigma_0,a\sigma_2\}$ is 1-dimensional, and let $\sigma$ be a non-zero element.
Then $\sigma\notin\Span_{\R}\{b\sigma_0,b\sigma_2\}$.  Then
\[ \mathcal{L}( h^{\beta_1,\xi_1} + h^{\beta_2,\xi_2} + h^{\beta_3,\xi_3},\sigma) =\mathcal{L}(h^{\beta_3,\xi_3},\sigma)\neq 0 \]
by Proposition~\ref{prop:inf-def-3Sasaki}.
\end{proof}

\begin{prop}\label{prop:3Sasak-def-ele}
Let $(M,g)$ be 3-Sasakian with $\dim M=4m-1$.  Suppose $\xi,\xi' \in\SR$ with $\xi \neq\xi'$ and $\xi \neq -\xi'$.
And suppose $\beta\in\mathcal{H}^1_{\mathcal{A}}(\xi)$ and $\beta'\in\mathcal{H}^1_{\mathcal{A}}(\xi')$ are non-zero,
then
\[ h^{\beta,\xi} + h^{\beta',\xi'} \in\EDS(g) \]
is non-zero and preserves a 1-dimensional subspace of Killing spinors if $m=2$ and no Killing spinors if $m>2$.
\end{prop}
\begin{proof}
We may suppose that $\xi =\xi_1$ and $\xi'=\cos(t)\xi_1 +\sin(t)\xi_2,\ 0<t<\pi,$ after possibly transforming by $\Sp(1)$.
Then $\xi' =\exp(\frac{t}{2}\pi)_* \xi_1$.  Set $a=\exp(\frac{t}{2}\pi)\in\Sp(1)$.
By Lemma~\ref{lem:ind-def} $h^{\beta,\xi} + h^{\beta',\xi'}\neq 0$.
Set $h_1 =h^{\beta,\xi}$ and $h^{\beta',\xi'} =ah_2$ with $h_2 \in \mathcal{H}^1_{\mathcal{A}}(\xi_1)$.
Suppose
\begin{equation}
\begin{split}
0=\mathcal{L}(h^{\beta,\xi} + h^{\beta',\xi'},\sigma)(X) & =\mathcal{L}(h_1 +ah_2 ,\sigma)(X) \\
                                & =\mathcal{L}(h,\sigma)(X) +a\mathcal{L}(h_2,a^{-1}\sigma)(a^{-1}\sigma X).
\end{split}
\end{equation}
The component of interest in this is given by (\ref{eq:def-comp}) which is
\begin{equation}\label{eq:def-ele}
-\Phi_1 h(X)\xi_1 \cdot \sigma -h(X)\partial_r \cdot\sigma -\Phi' h'(X)\xi'\cdot\sigma-h'(X)\partial_r\cdot\sigma,
\end{equation}
where for shorthand $h=h_1,\ h'=h^{\beta',\xi'}$ and $\Phi' =\cos(t)\Phi_1 +\sin(t)\Phi_2$.  Here
$\sigma =c_0\sigma_0 +\cdots +c_m \sigma_m$ is an arbitrary Killing spinor.

We consider the case $m>2$ first.  We compute (\ref{eq:def-ele}) using the notation in the proof of
Proposition~\ref{prop:inf-def-3Sasaki}.  In particular,
\[ \sigma_k =\frac{1}{k!} \vartheta^k +\frac{1}{(k-1)!}\vartheta^{k-1} \wedge\varepsilon_{2m-1}\wedge\varepsilon_{2m}. \]

The first two terms of (\ref{eq:def-ele}) with $\sigma=\sigma_k$ are
\begin{equation}
\frac{2\sqrt{2}\sqrt{-1}}{(k-1)!}\bigl(\vartheta^{k-1}\wedge\Phi_2 h(X)^{0,1}\wedge\varepsilon_{2m}+ \vartheta^{k-1}\wedge h(X)^{1,0}
\wedge\varepsilon_{2m-1} \bigr).
\end{equation}
The second two terms are
\begin{equation}
\begin{split}
&-\bigl(\cos(t)\Phi_1 h'(X)+\sin(t)\Phi_2 h'(X)\bigr)\bigl(\frac{-\cos(t)}{\sqrt{2}}(\varepsilon{2m}+\varepsilon{\ol{2m}})
 +\frac{\sqrt{-1}\sin(t)}{\sqrt{2}}(\varepsilon_{2m-1}-\varepsilon_{\ol{2m-1}}) \bigr) \sigma_k \\
& -h'(X)\frac{\sqrt{-1}}{\sqrt{2}}(\varepsilon_{2m} -\varepsilon_{\ol{2m}})\sigma_k.
\end{split}
\end{equation}

After a routine computation we get that (\ref{eq:def-ele}) with $\sigma=\sigma_k$ is
\[\begin{split}
&-\frac{\sqrt{-1}\sqrt{2}\sin^2(t)}{k!}h'(X)^{1,0}\wedge\varepsilon_{2m}\vartheta^k
+\frac{\sqrt{-1}2\sqrt{2}\cos^2(t)}{(k-1)!}\wedge \Phi_2 h'(X)^{1,0}\wedge\varepsilon_{2m}\wedge\vartheta^{k-1} \\
& +\frac{\sqrt{-1}2\sqrt{2}\cos^2(t)}{(k-1)!}h'(X)^{1,0}\wedge\varepsilon_{2m-1}\vartheta^{k-1}
-\frac{\sqrt{-1}\sqrt{2}\sin^2(t)}{(k-2)!}\Phi_2 h'(X)^{1,0}\wedge\varepsilon_{2m-1}\wedge\vartheta^{k-2} \\
&+\frac{\sqrt{2}\sin(t)\cos(t)}{k!}h'(X)^{1,0}\varepsilon_{2m-1}\vartheta^k
+\frac{2\sqrt{2}\sin(t)\cos(t)}{(k-1)!}\Phi_2 h'(X)^{1,0}\wedge\varepsilon_{2m-1}\wedge\vartheta^{k-1} \\
&+\frac{2\sqrt{2}\sin(t)\cos(t)}{(k-1)!}h'(X)^{1,0}\wedge\varepsilon_{2m}\wedge\vartheta^{k-1}
+\frac{\sqrt{2}\sin(t)\cos(t)}{(k-2)!}\Phi_2 h'(X)^{1,0}\wedge\varepsilon_{2m}\wedge\vartheta^{k-2} \\
&+\frac{\sqrt{2}\sin(t)\cos(t)}{k!}\Phi_2 h'(X)^{1,0}\wedge\varepsilon_{2m}\wedge\vartheta^k
+\frac{\sqrt{2}\sin(t)\cos(t)}{(k-2)!}h'(X)^{1,0}\wedge\vartheta_{2m-1}\wedge\vartheta^{k-2} \\
&-\frac{\sqrt{-1}\sqrt{2}\sin^2(t)}{k!}\Phi_2 h'(X)^{1,0}\wedge\varepsilon_{2m-1}\wedge\vartheta^k
-\frac{\sqrt{-1}\sqrt{2}\sin^2(t)}{(k-2)!}h'(X)^{1,0}\wedge\varepsilon_2m \wedge\vartheta^{k-2} \\
&+\frac{\sqrt{-1}2\sqrt{2}}{(k-1)!}\Phi_2 h(X)^{1,0}\wedge\varepsilon_{2m}\wedge\vartheta^{k-1}
+\frac{\sqrt{-1}2\sqrt{2}}{(k-1)!}h(X)^{1,0}\wedge\varepsilon_{2m-1}\wedge\vartheta^{k-1}.\\
\end{split}\]
Consider the image of a general Killing spinor $\sigma =c_0\sigma_0 +\cdots +c_m \sigma_m$ under (\ref{eq:def-ele}).
In particular, consider its component of degree $2k+2$ given by this formula for $0\leq k\leq m-2$.
From the $\varepsilon_2m$ and $\varepsilon_{2m-1}$ components we get the following equations after some manipulation:
\[ \begin{split}
0 & =c_k(\sqrt{2}\sin(t)\Phi' h'(X)) \\
  & +c_{k+1}(2\sqrt{2}\cos^2(t)h'(X) -2\sqrt{2}\sin(t)\cos(t)\Phi_3 h'(X) +2\sqrt{2}h(X)) \\
  & +c_{k+2}(\sqrt{2}\sin(t)\Phi' h'(X)) \\
\end{split}\]
and
\[ \begin{split}
0 & =c_k(\sqrt{2}\sin(t)\Phi' h'(X)) \\
  & +c_{k+1}(-2\sqrt{2}\cos^2(t)h'(X) +2\sqrt{2}\sin(t)\cos(t)\Phi_3 h'(X) -2\sqrt{2}h(X)) \\
  & +c_{k+2}(\sqrt{2}\sin(t)\Phi' h'(X)) \\
\end{split}\]
From these we get $c_k + c_{k+2} =0$ and $c_{k+1}(\cos(t)\Phi' h'(X) +\Phi_1 h(X))=0$, which implies $c_{k+1}=0$
from Lemma~\ref{lem:ind-def}.  This implies $\sigma=0$ when $m>2$.

If $m=2$ then we have $c_1 =0$ and $c_0 +c_2 =0$.  So the only possible Killing spinors preserved by
$h^{\beta,\xi} + h^{\beta',\xi'}$ are spanned by the real spinor $\varsigma_2$.  And one easily sees that
$\mathcal{L}(h^{\beta',\xi'},\varsigma_2)=0$ since $\exp(tk)\varsigma_2 =\varsigma_2$.
\end{proof}

Recall that $\gamma_2$ is the real representation of $\Sp(1)$, and
easy calculation shows that the standard basis of $\LSp(1)$ acts as follows in the basis $\varsigma_0,\varsigma_1,\varsigma_2$
\[ i=\begin{bmatrix}0&0&2\\0&0&0\\-2&0&0\end{bmatrix},j=\begin{bmatrix}0&0&0\\0&0&2\\0&-2&0\end{bmatrix},
k=\begin{bmatrix}0&-2&0\\2&0&0\\0&0&0\end{bmatrix}.    \]

\begin{prop}\label{prop:3Sasak-def-7dim}
Let $(M,g)$ be a 7-dimensional 3-Sasakian manifold.  Suppose $\xi_1,\xi_2,\xi_3 \in\SR$ are linearly independent and
$\beta_k\in\mathcal{H}^1_{\mathcal{A}}(\xi_k)\ k=1,2,3$ are each non-nonzero.  Then
\[ h^{\beta_1,\xi_1} + h^{\beta_2,\xi_2} +h^{\beta_3,\xi_3} \in\EDS(g) \]
is non-zero and preserves no Killing spinors
\end{prop}
\begin{proof}
By Lemma~\ref{lem:ind-def} $h^{\beta_1,\xi_1} + h^{\beta_2,\xi_2} +h^{\beta_3,\xi_3}$ is non-zero, so
we need to show it preserves no Killing spinors.

For simplicity we assume that $\xi_k, k=1,2,3$ is an orthonormal basis, which we may assume to be the standard basis after
a possibly acting by $\Sp(1)$.
By considering the $\Sp(1)$-action on $\gamma_2$, we see that $\mathcal{H}^1_{\mathcal{A}}(\xi_2)$ preserves
$\varsigma_1,\varsigma_2$ and $\mathcal{H}^1_{\mathcal{A}}(\xi_3)$ preserves $\varsigma_0,\varsigma_1$.
Let $\sigma=c_0\varsigma_0 +c_1 \varsigma_1 +c_2 \varsigma_2$, and denote $h^{\xi_k} =h^{\beta_k,\xi_k}$
Then suppose
\begin{equation}
\begin{split}
0 = & \mathcal{L}(h^{\beta_1,\xi_1} + h^{\beta_2,\xi_2} +h^{\beta_3,\xi_3},\sigma) \\
  = & c_1\mathcal{L}(h^{\beta_1,\xi_1},\varsigma_1)+c_0\mathcal{L}(h^{\beta_2,\xi_2},\varsigma_0)
  +c_2\mathcal{L}(h^{\beta_3,\xi_3},\varsigma_2) \\
  = & -c_1(\Phi_1 h^{\xi_1}(X)\xi_1 \cdot\varsigma_1 +h^{\xi_1}(X)\partial_r \cdot\varsigma_1) \\
    & - c_0(\Phi_2 h^{\xi_2}(X)\xi_2 \cdot\varsigma_0 +h^{\xi_2}(X)\partial_r \cdot\varsigma_0) \\
    & -c_2(\Phi_3 h^{\xi_3}(X)\xi_3 \cdot\varsigma_2 +h^{\xi_3}(X)\partial_r \cdot\varsigma_2). \\
\end{split}
\end{equation}
Routine calculation gives
\[ \begin{split}
\Phi_1 h^{\xi_1}(X)\xi_1 \cdot\varsigma_1 +h^{\xi_1}(X)\partial_r \cdot\varsigma_1 & =
2\sqrt{2}(\Phi_2 h^{\xi_1}(X)^{1,0} \wedge\varepsilon_4 +h^{\xi_1}(X)^{1,0} \wedge\varepsilon_3) \\
 \Phi_2 h^{\xi_2}(X)\xi_2 \cdot\varsigma_0 +h^{\xi_2}(X)\partial_r \cdot\varsigma_0 & =
2\sqrt{2}\sqrt{-1}(\Phi_2 h^{\xi_2}(X)^{1,0} \wedge\varepsilon_3 +h^{\xi_2}(X)^{1,0}\wedge\varepsilon_4) \\
\Phi_3 h^{\xi_3}(X)\xi_3 \cdot\varsigma_2 +h^{\xi_3}(X)\partial_r \cdot\varsigma_2  & =
2\sqrt{2}(\Phi_2 h^{\xi_3}(X)^{1,0} \wedge\varepsilon_3 -h^{\xi_3}(X)^{1,0} \wedge\varepsilon_4) \\
\end{split} \]
Thus we have
\[\begin{split}
0 = & -c_1 2\sqrt{2}(\Phi_2 h^{\xi_1}(X)^{1,0} \wedge\varepsilon_4 +h^{\xi_1}(X)^{1,0} \wedge\varepsilon_3) \\
    & -c_0 2\sqrt{2}\sqrt{-1}(\Phi_2 h^{\xi_2}(X)^{1,0} \wedge\varepsilon_3 +h^{\xi_2}(X)^{1,0}\wedge\varepsilon_4) \\
    & -c_2 2\sqrt{2}(\Phi_2 h^{\xi_3}(X)^{1,0} \wedge\varepsilon_3 -h^{\xi_3}(X)^{1,0} \wedge\varepsilon_4). \\
\end{split}\]
The $\varepsilon_3$ component gives
\[ c_1 \Phi_1 h^{\xi_1}(X) -c_0 \Phi_2 h^{\xi_2}(X) +c_2 \Phi_3 h^{\xi_3}(X) =0.\]
Lemma~\ref{lem:ind-def} now implies that $c_0 =c_1 =c_2 =0$.
\end{proof}

This proves Corollary~\ref{corint:Einst-def-3S}.  By Theorem~\ref{thm:3Sasak-int-def} for any
$\beta\in\re \mathcal{H}^1_{\mathcal{A}}(\xi)$ the deformation $h^{\beta,\xi}$ is integrable.
By Proposition~\ref{prop:3Sasak-def-ele} for $m>2$, and Proposition~\ref{prop:3Sasak-def-7dim} for $m=2$
there are elements in the span of these elements preserving no Killing spinors.

\subsection{Toric 3-Sasakian manifolds}\label{subsec:examples}

The examples of toric 3-Sasakian 7-manifolds from~\cite{BGMR98} provide interesting examples of Einstein deformations, integrable
and infinitesimal, preserving various numbers of Killing spinors.  This will give non-trivial examples of the theorems
of the previous sections.

\begin{defn}
A 3-Sasakian manifold $(M,g),\ \dim M=4m-1,$ is \emph{toric} if there is a $T^m \subseteq\Aut(M,g,\xi_1,\xi_2,\xi_3)$.
\end{defn}
\begin{remark}
Note that a toric 3-Sasakian manifold is generally not toric as a Sasakian manifold.

The isometry group of a 3-Sasakian manifold is
\[ \Aut(M,g,\xi_1,\xi_2,\xi_3)\times\Sp(1)\ \text{ or }\ \Aut(M,g,\xi_1,\xi_2,\xi_3)\times\SO(3),\]
where the $\Sp(1)$ or $\SO(3)$ factor is generated by the Reeb vector fields.
\end{remark}

Toric 3-Sasakian manifolds have been constructed from 3-Sasakian quotients by torus actions on $S^{4n-1}$~\cite{ BoyGal99,BGMR98},
with the 3-Sasakian structure given by right multiplication by $\Sp(1)$.  A subtorus $T^k \subset T^n$ is determined by a weight matrix $\Omega_{k,n}\in\operatorname{Mat}(k,n,\Z)$.  There are conditions on
$\Omega$, C. Boyer, K. Galicki, B. Mann, E. Rees, 1998~\cite{BGMR98}, that imply the moment map
\[ \mu :S^{4n-1} \rightarrow(\mathfrak{t}^k)^*\otimes\R^3 \]
is a submersion, and further that the quotient
\[ M_{\Omega_{k,n}} = S^{4n-1}/\negthickspace/{T^k} =\mu^{-1}(0)/{T^k} \]
is smooth.   When $n=k+2$ the above authors showed there are infinitely many weight matrices in $\operatorname{Mat}(k,n,\Z)$
for $k\geq 1$ giving infinitely many 7-manifolds $M_{\Omega_{k,n}}$ for each $b_2 =k\geq 1$.

\begin{lem}[\cite{vanCo06}]
Let $Z$ be the twistor space of a toric 3-Sasakian 7-manifold $M$, then $H^1(Z,\Theta_{Z})=H^1(Z,\Theta_{Z})^{T^2}$ and
\[\dim_\C H^1(Z,\Theta_{Z})= b_2(M)-1=k-1.\]
Thus $Z$ has a local $b_2(M)-1$-dimensional space of deformations.
\end{lem}

If $b_2(M)\geq 1$, then the maximal torus of Sasakian automorphisms, $T^3 \subset\Aut(M,\xi_1)$, is 3-dimensional.
Theorem~\ref{thm:def-Sasak-Einst} implies the following.
\begin{thm}\label{thm:toric-3-Sasak-def}
Let $(M,g)$ be a toric 3-Sasakian 7-manifold.  Then $(M,g)$ has a 3-dimensional space of Killing spinors spanned by
$\sigma_0,\sigma_1,\sigma_2$.  Then $g$ is in an effective complex $b_2(M)-1$-dimensional family
$\{ g_t\}_{t\in\mathcal{U}},\ \mathcal{U}\subset\C^{b_2(M)-1}$ with $g_0 =g$, of Sasaki-Einstein metrics where $g_t$ is not
3-Sasakian for $t\neq 0$.

Therefore the deformations preserve a two dimensional subspace of Killing spinors spanned by $\sigma_0, \sigma_2$.
\end{thm}
The deformation space of Sasaki-Einstein metrics with their isometry groups is illustrated in Figure~\ref{fig:def-space}.

\begin{figure}[t!]
\labellist
\hair 2pt
\pinlabel $\C^{b_2 -1}$ at 241 234
\pinlabel $\R^{b_2-1}$ [l] at 278 142
\pinlabel $T^3$ at 113 181
\pinlabel {$T^3 \rtimes\Z_2$} [r] at 40 109
\pinlabel {$T^2 \times\Sp(1)$} [l] at 173 145
\endlabellist
\centering
\includegraphics[scale=.5]{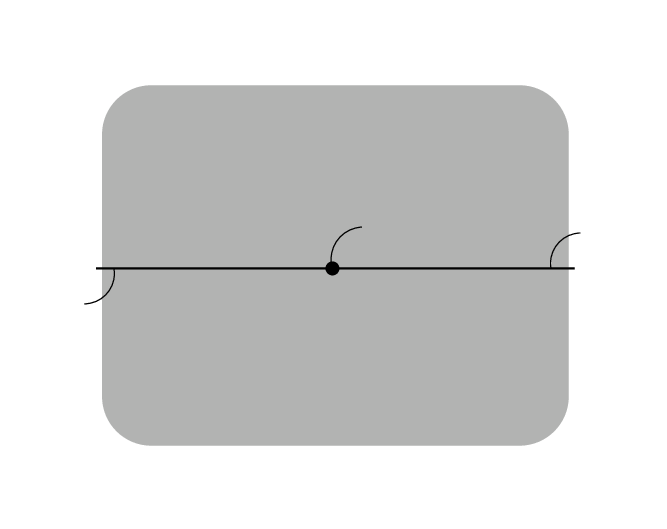}
\caption{Space of Sasaki-Einstein metrics}
\label{fig:def-space}
\end{figure}
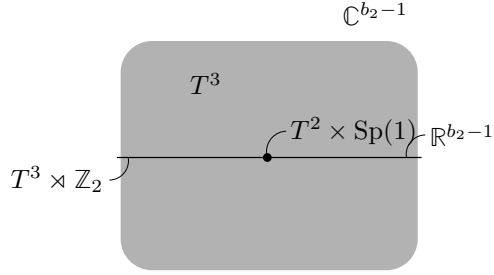

For a given $\xi\in\SR$, the space of infinitesimal Einstein deformations
$\{ h^{\beta,\xi}\ |\ \beta\in\mathcal{H}^1_{\mathcal{A}}(\xi)\} \subseteq\EED(g)$
integrate to Einstein deformations preserving Killing spinors $\sigma_0$ and $\sigma_2$ but not $\sigma_1$.
Note that the space $\EDS(g)$ defined in (\ref{eq:3Sasak-def-space}) is spanned by integrable Einstein deformations.
Theorem~\ref{thmint: inf-Einst-tor-3S} now follows from Proposition~\ref{prop:3Sasak-def-ele} and Proposition~\ref{prop:3Sasak-def-7dim}.

\bibliographystyle{amsplain}

\providecommand{\bysame}{\leavevmode\hbox to3em{\hrulefill}\thinspace}
\providecommand{\MR}{\relax\ifhmode\unskip\space\fi MR }
% \MRhref is called by the amsart/book/proc definition of \MR.
\providecommand{\MRhref}[2]{%
  \href{http://www.ams.org/mathscinet-getitem?mr=#1}{#2}
}
\providecommand{\href}[2]{#2}

\end{document}